\documentclass[a4paper,oneside,10.5pt]{article}%
\usepackage{amsmath}
\usepackage{amsfonts}
\usepackage{amssymb}
\usepackage{graphicx}%
\providecommand{\U}[1]{\protect\rule{.1in}{.1in}}

\newtheorem{theorem}{Theorem}[section]

\newtheorem{corollary}[theorem]{Corollary}

\newtheorem{definition}[theorem]{Definition}
\newtheorem{assumption}[theorem]{Assumption}
\newtheorem{example}[theorem]{Example}

\newtheorem{lemma}[theorem]{Lemma}

\newtheorem{proposition}[theorem]{Proposition}
\newtheorem{remark}[theorem]{Remark}

\newenvironment{proof}[1][Proof]{\noindent \textbf{#1.} }{\  \rule{0.5em}{0.5em}}
\numberwithin{equation}{section}

\begin{document}

\title{Solvability of finite state forward-backward stochastic difference equations}
\author{Shaolin Ji\thanks{Zhongtai Securities Institute for Financial Studies,
Shandong University, Jinan, Shandong 250100, PR China. jsl@sdu.edu.cn.
Research supported by NSF (No. 11571203).}
\and Haodong Liu\thanks{Zhongtai Securities Institute for Financial Studies,
Shandong University, Jinan, Shandong 250100, PR China. (Corresponding
author).}}
\maketitle

\textbf{Abstract}: In this paper, we consider the solvability problems for the
fully coupled forward-backward stochastic difference equations (FBS$\Delta$Es)
on spaces related to discrete time, finite state processes. On one hand, we
provide the necessary and sufficient condition for the solvability of the
linear FBS$\Delta$Es. On the other hand, under the assumption that the
coefficients satisfy the monotone condition, we investigate the existence and
uniqueness theorems for the general nonlinear FBS$\Delta$Es.


\addcontentsline{toc}{section}{\hspace*{1.8em}Abstract}

\section{Introduction}

It is well-known that forward-backward stochastic differential equations
(FBSDEs) are widely studied by many researchers and there are fruitful results
in both theory and application (see \cite{hp95} and
\cite{mpy94,mwzz15,my93,pt99}).

In this paper, we study forward-backward stochastic difference equations
(FBS$\Delta$Es) on spaces related to discrete time, finite state processes. As
the discrete time counterpart of FBSDEs, the FBS$\Delta$Es also have wide
applications in many areas, such as discrete time stochastic optimal control
theory, stochastic difference games, mathematical finance, etc. For instance,
the Hamiltonian systems of discrete time stochastic optimal control problems
are FBS$\Delta$Es (see \cite{lz15}, \cite{zlxf15}); FBS$\Delta$Es can also be
regarded as the state equations for a discrete time recursive utility
optimization problem since some discrete time nonlinear expectations are
defined by backward stochastic difference equations (BS$\Delta$Es). As far as
we know, there are few works dealing with the solvability of FBS$\Delta$Es
because the forward variables and backward variables are coupled in these equations.

The starting point of exploring the solvability of FBS$\Delta$Es is to
formulate FBS$\Delta$Es. It is worth pointing out that even though the
BS$\Delta$E is the discrete time counterpart of the backward stochastic
differential equation (BSDE), the formulations are quite different. Based on
the driving process, there are mainly two formulations of BS$\Delta$Es (see,
e.g. \cite{bcc14}, \cite{ce10+}, \cite{ce11}). One is driving by a finite
state process taking values from the basis vectors (as in \cite{ce10+}) and
the other is driving by a martingale with independent increments (as in
\cite{bcc14}). For the first formulation, the researchers in \cite{ce10+}
obtained the discrete time version of martingale representation theorem and
establish the solvability result of BS$\Delta$E with the uniqueness of $Z$
under a new kind of equivalence relation.\ Further works about the
applications of this formulation can be seen in \cite{em11}, \cite{m10},
\cite{ly14}. In this paper, we take this formulation to construct our
FBS$\Delta$Es. Note that taking the second formulation, one kind of linear
FBS$\Delta$Es was studied in \cite{xzx17}, \cite{xzx18}.

Before investigating the solvability of FBS$\Delta$Es, as preliminaries, we
prove some results about the BS$\Delta$Es. In more details, we formulate one
kind of BS$\Delta$Es in our context, obtain an explicit representation of the
solution $(Z_{t})$ and prove the existence and uniqueness of solutions to our
formulated BS$\Delta$E. Note that the driving process $M$ usually doesn't have
independent increments which leads to the estimation of the solution $Z_{t}$
can not be derived directly from the value of $\mathbb{E}\left\vert
Z_{t}M_{t+1}\right\vert ^{2}$. But in our context, we overcome this problem
and show that the norm of $Z_{t}$ is dominated by $L\mathbb{E}\left\vert
Z_{t}M_{t+1}\right\vert ^{2}$ where the constant $L$ depends on the
probability $P$.

Then we study the solvability problem for the linear FBS$\Delta$Es. As the
special case of the general nonlinear FBS$\Delta$Es, the simple and nice
structure of the linear FBS$\Delta$Es leads to a necessary and sufficient
condition for the existence and uniqueness of the solutions. On the other
hand, the obtained results for the linear FBS$\Delta$Es will be applied to
prove the solvability results for the nonlinear case. For the linear case,
FBS$\Delta$Es can be transformed into $N$-dimensional linear algebraic
equations of $X_{t+1}$. Based on this, we prove the equivalence between the
solvability of linear FBS$\Delta$E and the solvability of linear algebraic
equations which leads to a necessary and sufficient condition for the
existence and uniqueness of solutions. By solving the linear algebraic
equations, we can decouple the forward and backward variables and obtain the
explicit expression of $Y_{t}$ with respect to $X_{t}$. Thus, the linear
FBS$\Delta$E can be solved recursively in an explicit form.

Finally, the solvability problem for the nonlinear FBS$\Delta$Es is studied.
We apply the continuation method developed in \cite{hp95} to our discrete time
framework. Under the assumption that the coefficients satisfy the monotone
condition, we obtain the existence and uniqueness theorem for the general
nonlinear FBS$\Delta$Es. There are mainly two techniques which are different
from the continuous time case when we apply the continuation method in our
discrete time context. Firstly, It\^{o} formula, which is the basic technique
in the continuous time stochastic calculus, doesn't work in our discrete time
case. So we must choose a suitable representation of the product rule
$\Delta\left\langle X_{t},Y_{t}\right\rangle =\left\langle \Delta X_{t}%
,Y_{t}\right\rangle +\left\langle X_{t},\Delta Y_{t}\right\rangle
+\left\langle \Delta X_{t},\Delta Y_{t}\right\rangle $. Here we propose the
following form of the product rule%
\[
\Delta\left\langle X_{t},Y_{t}\right\rangle =\left\langle X_{t+1},\Delta
Y_{t}\right\rangle +\left\langle \Delta X_{t},Y_{t}\right\rangle
\]
in order to obtain the desired result. Secondly, since the driving process $M$
doesn't have independent increments, we suggest a new monotone condition which
contains the $\mathbb{E}\left[  M_{t+1}M_{t+1}^{\ast}|\mathcal{F}_{t}\right]
$ term.

The remainder of this paper is organized as follows. In Section 2 we present
preliminary results of the BS$\Delta$Es. We then obtain the solvability
results for linear FBS$\Delta$E (\ref{fbsde_linear_1d}) in Section 3. The
solvability results for nonlinear FBS$\Delta$E
(\ref{nonlinear_fbsde_c_method_m0}) are given in Section 4.

\section{Preliminaries}

Following \cite{ce10+}, we consider an underlying discrete time, finite state
process $W$ which takes values in the standard basis vectors of $\mathbb{R}%
^{N}$, where $N$ is the number of states of the process $W$. In more detail,
for each $t\in\left\{  0,1,...,T\right\}  $, $W_{t}\in\left\{  e_{1}%
,e_{2},...e_{N}\right\}  $, where $T>0$ is a finite deterministic terminal
time, $e_{i}=\left(  0,0,...,0,1,0,...,0\right)  ^{\ast}\in\mathbb{R}^{N}$,
and $\left[  \cdot\right]  ^{\ast}$ denotes vector transposition. Consider a
filtered probability space $\left(  \Omega,\mathcal{F},\left\{  \mathcal{F}%
_{t}\right\}  _{0\leq t\leq T},P\right)  $, where $\mathcal{F}_{t}$ is the
completion of the $\sigma$-algebra generated by the process $W$ up to time $t$
and $\mathcal{F}=\mathcal{F}_{T}$.

For simplicity, we suppose the process $W$ satisfies the following assumption.

\begin{assumption}
\label{general_assumption}For any $t\in\left\{  0,1,2,...,T-1\right\}  $ and
$\omega\in\Omega$, $\mathbb{E}\left[  W_{t+1}|\mathcal{F}_{t}\right]  \left(
\omega\right)  >0.$
\end{assumption}

Note that in this paper, an inequality on a vector quantity is to hold
componentwise. Under the above assumption, the conception "$P-$almost surely"
in the following statements can be changed to "for every $\omega$". In fact,
this assumption is not necessary and is just for the sake of simple expression.

Denote by $L\left(  \mathcal{F}_{t};\mathbb{R}^{K\times N}\right)  $ the set
of all $\mathcal{F}_{t}-$adapted random variable $X_{t}$ taking values in
$\mathbb{R}^{K\times N}$ and by $\mathcal{M}\left(  0,t;\mathbb{R}^{K\times
N}\right)  $ the set of all $\left\{  \mathcal{F}_{s}\right\}  _{0\leq s\leq
t}$ adapted process $X$ taking values in $\mathbb{R}^{K\times N}$ with the
norm $\left\Vert X\right\Vert =\left(  \mathbb{E}\left[  \sum_{s=0}%
^{t}\left\vert X_{s}\right\vert ^{2}\right]  \right)  ^{\frac{1}{2}}$. Let
$(P_{t}^{1},...,P_{t}^{N})^{\ast}:=\mathbb{E}\left[  W_{t+1}|\mathcal{F}%
_{t}\right]  $.

Define%
\[
M_{t}=W_{t}-\mathbb{E}\left[  W_{t}|\mathcal{F}_{t-1}\right]  ,t=1,...,T.
\]

Then $M$ is a martingale difference process taking values in $\mathbb{R}^{N}$.
The following definition is given in \cite{ce10+}.

\begin{definition}
For two $\mathcal{F}_{t}$-measurable random variables $Z_{t}$ and
$\widetilde{Z}_{t}$, we define $Z_{t}\thicksim_{M_{t+1}}\widetilde{Z}_{t}$, if
$Z_{t}M_{t+1}=\widetilde{Z}_{t}M_{t+1};$

For two adapted processes $Z$ and $\widetilde{Z}$, we define $Z\thicksim
_{M}\widetilde{Z}$, if $Z_{t}M_{t+1}=\widetilde{Z}_{t}M_{t+1}$ for any
$t\in\left\{  0,1,2,...,T-1\right\}  .$
\end{definition}

Denote by $\mathbf{1}_{N}\mathbf{=}\left(  1,1,...,1\right)  ^{\ast}$ the
$N$-dimensional vector where every element is equal to $1$, by $\mathbf{1}%
_{N\times N}$ the $N\times N$-matrix where every element is equal to $1$, by
$I_{N\times N}$ the $N\times N$-identity matrix and by $\widetilde{I}$ the
$N\times\left(  N-1\right)  $-matrix $%
\begin{pmatrix}
I_{(N-1)\times\left(  N-1\right)  } & -\mathbf{1}_{N-1}%
\end{pmatrix}
^{\ast}$.

In the following lemma, we give three equivalent statements which are
different from the ones in \cite{ce10+}.

\begin{lemma}
\label{m equality}For $Z_{t}$ and $\widetilde{Z}_{t}\in L\left(
\mathcal{F}_{t};\mathbb{R}^{K\times N}\right)  $, the following statements are equivalent:

(i) $Z_{t}\thicksim_{M_{t+1}}\widetilde{Z}_{t};$

(ii) There exists $C_{t}\in L\left(  \mathcal{F}_{t};\mathbb{R}^{K}\right)  $
such that $Z_{t}=\widetilde{Z}_{t}+C_{t}\mathbf{1}_{N}^{\ast}\mathbf{;}$

(iii) $Z_{t}\widetilde{I}=\widetilde{Z}_{t}\widetilde{I}.$
\end{lemma}

\begin{proof}
Set $Z_{t}=\left(  Z_{t}^{1},Z_{t}^{2},...,Z_{t}^{K}\right)  ^{\ast}$ and
$\widetilde{Z}_{t}=\left(  \widetilde{Z}_{t}^{1},\widetilde{Z}_{t}%
^{2},...,\widetilde{Z}_{t}^{K}\right)  ^{\ast}$ where $Z_{t}^{i}=\left(
Z_{t}^{i1},Z_{t}^{i2},...,Z_{t}^{iN}\right)  $ and\ $\widetilde{Z}_{t}%
^{i}=\left(  \widetilde{Z}_{t}^{i1},\widetilde{Z}_{t}^{i2},...,\widetilde{Z}%
_{t}^{iN}\right)  $ for $i\in\left\{  1,2,...,K\right\}  $.

\textbf{Step 1.} (i) $\Rightarrow$ (ii):

If $Z_{t}\thicksim_{M_{t+1}}\widetilde{Z}_{t}$, then $Z_{t}M_{t+1}%
=\widetilde{Z}_{t}M_{t+1}$. For a given $i$, we have%
\[
\left(  W_{t+1}-\mathbb{E}\left[  W_{t+1}|\mathcal{F}_{t}\right]  \right)
^{\ast}\left(  Z_{t}^{i}-\widetilde{Z}_{t}^{i}\right)  ^{\ast}=0.
\]
Then for every $j\in\left\{  1,2,...,N\right\}  $, by multiplying the above
equation with the indicator function $1_{\left\{  W_{t+1}=e_{j}\right\}  }$
and taking the $\mathcal{F}_{t}$-conditional expectation, we get the following
linear equations for $(Z_{t}^{i1}-\widetilde{Z}_{t}^{i1},Z_{t}^{i2}%
-\widetilde{Z}_{t}^{i2},\cdots,Z_{t}^{iN}-\widetilde{Z}_{t}^{iN})$:%
\[%
\begin{pmatrix}
e_{1}^{\ast}\mathbb{E}\left[  W_{t+1}|\mathcal{F}_{t}\right]  \left(
e_{1}-\mathbb{E}\left[  W_{t+1}|\mathcal{F}_{t}\right]  \right)  ^{\ast}\\
e_{2}^{\ast}\mathbb{E}\left[  W_{t+1}|\mathcal{F}_{t}\right]  \left(
e_{2}-\mathbb{E}\left[  W_{t+1}|\mathcal{F}_{t}\right]  \right)  ^{\ast}\\
\vdots\\
e_{N}^{\ast}\mathbb{E}\left[  W_{t+1}|\mathcal{F}_{t}\right]  \left(
e_{N}-\mathbb{E}\left[  W_{t+1}|\mathcal{F}_{t}\right]  \right)  ^{\ast}%
\end{pmatrix}%
\begin{pmatrix}
Z_{t}^{i1}-\widetilde{Z}_{t}^{i1}\\
Z_{t}^{i2}-\widetilde{Z}_{t}^{i2}\\
\vdots\\
Z_{t}^{iN}-\widetilde{Z}_{t}^{iN}%
\end{pmatrix}
=%
\begin{pmatrix}
0\\
0\\
\vdots\\
0
\end{pmatrix}
.
\]
By Assumption \ref{general_assumption}, $e_{j}^{\ast}\mathbb{E}\left[
W_{t+1}|\mathcal{F}_{t}\right]  \neq0,\,j\in\left\{  1,2,...,N\right\}  $.
Hence we have%
\[%
\begin{pmatrix}
1 & 0 & \cdots & 0\\
0 & 1 & \cdots & 0\\
\vdots & \vdots & \ddots & \vdots\\
0 & 0 & \cdots & 1
\end{pmatrix}%
\begin{pmatrix}
Z_{t}^{i1}-\widetilde{Z}_{t}^{i1}\\
Z_{t}^{i2}-\widetilde{Z}_{t}^{i2}\\
\vdots\\
Z_{t}^{iN}-\widetilde{Z}_{t}^{iN}%
\end{pmatrix}
=%
\begin{pmatrix}
\left(  \mathbb{E}\left[  W_{t+1}|\mathcal{F}_{t}\right]  \right)  ^{\ast}\\
\left(  \mathbb{E}\left[  W_{t+1}|\mathcal{F}_{t}\right]  \right)  ^{\ast}\\
\vdots\\
\left(  \mathbb{E}\left[  W_{t+1}|\mathcal{F}_{t}\right]  \right)  ^{\ast}%
\end{pmatrix}%
\begin{pmatrix}
Z_{t}^{i1}-\widetilde{Z}_{t}^{i1}\\
Z_{t}^{i2}-\widetilde{Z}_{t}^{i2}\\
\vdots\\
Z_{t}^{iN}-\widetilde{Z}_{t}^{iN}%
\end{pmatrix}
.
\]
Since $\sum_{j=1}^{N}e_{j}^{\ast}\mathbb{E}\left[  W_{t+1}|\mathcal{F}%
_{t}\right]  =1$, the solutions to the above equations are%
\[
Z_{t}^{i1}-\widetilde{Z}_{t}^{i1}=Z_{t}^{i2}-\widetilde{Z}_{t}^{i2}%
=\cdots=Z_{t}^{iN}-\widetilde{Z}_{t}^{iN}.
\]
Let\thinspace$C_{t}^{i}=Z_{t}^{i1}-\widetilde{Z}_{t}^{i1}$ and $C_{t}%
=\sum_{i=1}^{K}C_{t}^{i}e_{i}$. Then, we obtain $Z_{t}=\widetilde{Z}_{t}%
+C_{t}\mathbf{1}_{N}^{\ast}\mathbf{.}$

\textbf{Step 2.} (ii) $\Rightarrow$ (iii):

Suppose that for any $i\in\left\{  1,2,...,K\right\}  $,
\begin{equation}
Z_{t}^{i1}-\widetilde{Z}_{t}^{i1}=Z_{t}^{i2}-\widetilde{Z}_{t}^{i2}%
=\cdots=Z_{t}^{iN}-\widetilde{Z}_{t}^{iN}. \label{equivalent-1}%
\end{equation}
(\ref{equivalent-1}) can be rewritten as\ $Z_{t}^{ij}-Z_{t}^{iN}%
=\widetilde{Z}_{t}^{ij}-\widetilde{Z}_{t}^{iN},$ $j\in\left\{
1,2,...,N-1\right\}  $. Thus,%
\[%
\begin{pmatrix}
1 & 0 & \cdots & 0 & -1\\
0 & 1 & \cdots & 0 & -1\\
\vdots & \vdots & \ddots & \vdots & -1\\
0 & 0 & \cdots & 1 & -1
\end{pmatrix}%
\begin{pmatrix}
Z_{t}^{i1}\\
Z_{t}^{i2}\\
\vdots\\
Z_{t}^{id}%
\end{pmatrix}
=%
\begin{pmatrix}
1 & 0 & \cdots & 0 & -1\\
0 & 1 & \cdots & 0 & -1\\
\vdots & \vdots & \ddots & \vdots & -1\\
0 & 0 & \cdots & 1 & -1
\end{pmatrix}%
\begin{pmatrix}
\widetilde{Z}_{t}^{i1}\\
\widetilde{Z}_{t}^{i2}\\
\vdots\\
\widetilde{Z}_{t}^{id}%
\end{pmatrix}
.
\]
The result is obvious.

\textbf{Step 3.} (iii) $\Rightarrow$ (i):

If (iii) holds, then $Z_{t}^{i1}-\widetilde{Z}_{t}^{i1}=Z_{t}^{i2}%
-\widetilde{Z}_{t}^{i2}=\cdots=Z_{t}^{iN}-\widetilde{Z}_{t}^{iN},$ for
$i\in\left\{  1,2,...,K\right\}  $ which implies that $Z_{t}M_{t+1}%
=\widetilde{Z}_{t}M_{t+1}$. This completes the proof.
\end{proof}

Set%
\begin{align*}
\overline{L}  &  =\left(  N-1\right)  \max\limits_{t\in\left\{
0,...,T-1\right\}  ,k\in\left\{  1,...,N-1\right\}  ,w\in\Omega}\left(
1-P_{t}^{k}\left(  \omega\right)  \right)  P_{t}^{k}\left(  \omega\right)  ,\\
\underline{L}  &  =\frac{1}{2}\left(  N-1\right)  ^{-1}\min\limits_{t\in
\left\{  0,...,T-1\right\}  ,k\in\left\{  1,...,N\right\}  ,w\in\Omega}%
P_{t}^{k}\left(  \omega\right)  .
\end{align*}

\begin{proposition}
\label{norm equality}For any $Z_{t}\in L\left(  \mathcal{F}_{t};\mathbb{R}%
^{K\times N}\right)  $,
\[
\underline{L}\mathbb{E}\left[  \left\Vert Z_{t}\widetilde{I}\right\Vert
^{2}\right]  \leq\mathbb{E}\left[  \left\vert Z_{t}M_{t+1}\right\vert
^{2}\right]  \leq\overline{L}\mathbb{E}\left[  \left\Vert Z_{t}\widetilde{I}%
\right\Vert ^{2}\right]  .
\]

\end{proposition}

\begin{proof}
Without loss of generality, we only consider the case of $K=1$. For a fixed
$t\in\left\{  0,1,...,T-1\right\}  $ and for any given $Z_{t}\in L\left(
\mathcal{F}_{t};\mathbb{R}^{1\times N}\right)  $,\ define $\widetilde{Z}%
_{t}=\left(  \widetilde{Z}_{t}^{1},\widetilde{Z}_{t}^{2},...,\widetilde{Z}%
_{t}^{N-1},\widetilde{Z}_{t}^{N}\right)  $, where $\left(  \widetilde{Z}%
_{t}^{1},\widetilde{Z}_{t}^{2},...,\widetilde{Z}_{t}^{N-1}\right)
=Z_{t}\widetilde{I}$ and $\widetilde{Z}_{t}^{N}=0$. by Lemma \ref{m equality},
we have $Z_{t}\thicksim_{M_{t+1}}\widetilde{Z}_{t}$ which leads to
$Z_{t}M_{t+1}=\widetilde{Z}_{t}M_{t+1}$. Since $\widetilde{Z}_{t}^{N}=0$ and
$P_{t}^{i}=e_{i}^{\ast}\mathbb{E}\left[  W_{t+1}|\mathcal{F}_{t}\right]  $ for
$i\in\left\{  1,2,...,N\right\}  $, we have%
\[%
\begin{array}
[c]{cl}
& \mathbb{E}\left[  \left(  Z_{t}M_{t+1}\right)  ^{2}|\mathcal{F}_{t}\right]
\\
= & \mathbb{E}\left[  \left(  \widetilde{Z}_{t}M_{t+1}\right)  ^{2}%
|\mathcal{F}_{t}\right] \\
= & \sum_{i=1}^{N}e_{i}^{\ast}\mathbb{E}\left[  W_{t+1}|\mathcal{F}%
_{t}\right]  \left[  \widetilde{Z}_{t}\left(  e_{i}-\mathbb{E}\left[
W_{t+1}|\mathcal{F}_{t}\right]  \right)  \right]  ^{2}\\
= & \sum_{i=1}^{N-1}e_{i}^{\ast}\mathbb{E}\left[  W_{t+1}|\mathcal{F}%
_{t}\right]  \left[  \widetilde{Z}_{t}\left(  e_{i}-\mathbb{E}\left[
W_{t+1}|\mathcal{F}_{t}\right]  \right)  \right]  ^{2}+e_{N}^{\ast}%
\mathbb{E}\left[  W_{t+1}|\mathcal{F}_{t}\right]  \left(  \widetilde{Z}%
_{t}\mathbb{E}\left[  W_{t+1}|\mathcal{F}_{t}\right]  \right)  ^{2}\\
\leq & \sum\limits_{i=1}^{N-1}\left(  N-1\right)  P_{t}^{i}\left[  \left(
1-P_{t}^{i}\right)  ^{2}\left(  \widetilde{Z}_{t}^{i}\right)  ^{2}%
+\sum\limits_{j=1,j\neq i}^{N-1}\left(  P_{t}^{j}\right)  ^{2}\left(
\widetilde{Z}_{t}^{j}\right)  ^{2}\right]  +\left(  N-1\right)  P_{t}%
^{N}\left[  \sum\limits_{j=1}^{N-1}\left(  P_{t}^{j}\right)  ^{2}\left(
\widetilde{Z}_{t}^{j}\right)  ^{2}\right] \\
= & \left(  N-1\right)  \sum\limits_{k=1}^{N-1}\left[  P_{t}^{k}\left(
1-P_{t}^{k}\right)  ^{2}+\sum\limits_{i=1,i\neq k}^{N}P_{t}^{i}\left(
P_{t}^{k}\right)  ^{2}\right]  \left(  \widetilde{Z}_{t}^{k}\right)  ^{2}\\
= & \left(  N-1\right)  \sum\limits_{k=1}^{N-1}\left(  1-P_{t}^{k}\right)
P_{t}^{k}\left(  \widetilde{Z}_{t}^{k}\right)  ^{2}.
\end{array}
\]
Set
\[
\overline{L}_{t}=\left(  N-1\right)  \max_{k\in\left\{  1,2,...,d-1\right\}
,\omega\in\Omega}\left(  1-P_{t}^{k}\left(  \omega\right)  \right)  P_{t}%
^{k}\left(  \omega\right)  .
\]
Then, we obtain%
\[%
\begin{array}
[c]{cccl}
& \mathbb{E}\left[  \left(  Z_{t}M_{t+1}\right)  ^{2}\right]  & = &
\mathbb{E}\left[  \left(  \widetilde{Z}_{t}M_{t+1}\right)  ^{2}\right]
=\mathbb{E}\left[  \mathbb{E}\left[  \left(  \widetilde{Z}_{t}M_{t+1}\right)
^{2}|\mathcal{F}_{t}\right]  \right] \\
\leq & \overline{L}_{t}\mathbb{E}\left[  \left\vert \widetilde{Z}%
_{t}\right\vert ^{2}\right]  & = & \overline{L}_{t}\mathbb{E}\left[
\left\vert Z_{t}\widetilde{I}\right\vert ^{2}\right]  .
\end{array}
\]

In order to prove that there exists a constant $\underline{L}_{t}>0$ such that
$\mathbb{E}\left[  \left(  Z_{t}M_{t+1}\right)  ^{2}\right]  \geq
\underline{L}_{t}\mathbb{E}\left[  \left\vert Z_{t}\widetilde{I}\right\vert
^{2}\right]  $. We introduce variables $\xi_{i}$ for $i\in\left\{
1,2,...,N-1\right\}  $ such that%
\[%
\begin{pmatrix}
\left(  1-P_{t}^{1}\right)  & -P_{t}^{2} & \cdots & -P_{t}^{N-1}\\
-P_{t}^{1} & \left(  1-P_{t}^{2}\right)  & \cdots & -P_{t}^{N-1}\\
\vdots & \vdots & \ddots & \vdots\\
-P_{t}^{1} & -P_{t}^{2} & \cdots & 1-P_{t}^{N-1}%
\end{pmatrix}%
\begin{pmatrix}
\widetilde{Z}_{t}^{1}\\
\widetilde{Z}_{t}^{2}\\
\vdots\\
\widetilde{Z}_{t}^{N-1}%
\end{pmatrix}
=%
\begin{pmatrix}
\xi_{1}\\
\xi_{2}\\
\vdots\\
\xi_{N-1}%
\end{pmatrix}
.
\]
For any $k\in\left\{  1,2,...,N-1\right\}  $, $\sum_{j=1}^{N-1}P_{t}%
^{j}\widetilde{Z}_{t}^{j}=\widetilde{Z}_{t}^{k}-\xi_{k}$. Then we have%
\[%
\begin{array}
[c]{cl}
& \mathbb{E}\left[  \left(  \widetilde{Z}_{t}M_{t+1}\right)  ^{2}%
|\mathcal{F}_{t}\right] \\
= & \sum\limits_{i=1}^{N-1}e_{i}^{\ast}\mathbb{E}\left[  W_{t+1}%
|\mathcal{F}_{t}\right]  \left[  \widetilde{Z}_{t}\left(  e_{i}-\mathbb{E}%
\left[  W_{t+1}|\mathcal{F}_{t}\right]  \right)  \right]  ^{2}+e_{N}^{\ast
}\mathbb{E}\left[  W_{t+1}|\mathcal{F}_{t}\right]  \left(  \widetilde{Z}%
_{t}\mathbb{E}\left[  W_{t+1}|\mathcal{F}_{t}\right]  \right)  ^{2}\\
= & \sum\limits_{i=1,i\neq k}^{N-1}P_{t}^{i}\xi_{i}^{2}+P_{t}^{k}\xi_{k}%
^{2}+P_{t}^{N}\left(  \widetilde{Z}_{t}^{k}-\xi_{k}\right)  ^{2}\\
\geq & \sum\limits_{i=1,i\neq k}^{N-1}P_{t}^{i}\xi_{i}^{2}+\min\left\{
P_{t}^{k},P_{t}^{N}\right\}  \times\left[  \xi_{k}^{2}+\left(  \widetilde{Z}%
_{t}^{k}-\xi_{k}\right)  ^{2}\right] \\
\geq & \min\left\{  P_{t}^{k},P_{t}^{N}\right\}  \times\frac{1}{2}\left(
\widetilde{Z}_{t}^{k}\right)  ^{2}.\\
\geq & \frac{1}{2}\left(  N-1\right)  ^{-1}\sum\limits_{k=1}^{N-1}\min\left\{
P_{t}^{k},P_{t}^{N}\right\}  \times\left(  \widetilde{Z}_{t}^{k}\right)
^{2}\\
\geq & \frac{1}{2}\left(  N-1\right)  ^{-1}\min\limits_{k\in\left\{
1,2,...,N\right\}  }\left\{  P_{t}^{k}\right\}  \times\sum\limits_{k=1}%
^{N-1}\left(  \widetilde{Z}_{t}^{k}\right)  ^{2}.
\end{array}
\]
Let%
\[
\underline{L}_{t}=\frac{1}{2}\left(  N-1\right)  ^{-1}\min_{k\in\left\{
1,2,...,N\right\}  ,w\in\Omega}P_{t}^{k}\left(  \omega\right)  ,
\]
It is obvious that $\underline{L}_{t}>0$ and $\mathbb{E}\left[  \left(
Z_{t}M_{t+1}\right)  ^{2}\right]  \geq\underline{L}_{t}\mathbb{E}\left[
\left\vert Z_{t}\widetilde{I}\right\vert ^{2}\right]  $. Thus, we have%
\[
\underline{L}\mathbb{E}\left[  \left\vert Z_{t}\widetilde{I}\right\vert
^{2}\right]  \leq\mathbb{E}\left[  \left(  Z_{t}M_{t+1}\right)  ^{2}\right]
\leq\,\overline{L}\mathbb{E}\left[  \left\vert Z_{t}\widetilde{I}\right\vert
^{2}\right]  .
\]
This completes the proof.
\end{proof}

The following representation theorem is from Corollary 1 in \cite{ce10+}.

\begin{theorem}
\label{mrt_result_cohen_c}For any $\mathbb{R}^{K}$ valued and $\mathcal{F}%
_{t+1}$-measurable random variable $Y$, there exists an $\mathcal{F}_{t}%
$-measurable random variable $Z_{t}\in L\left(  \mathcal{F}_{t};\mathbb{R}%
^{K\times N}\right)  $ such that $Y-\mathbb{E}\left[  Y|\mathcal{F}%
_{t}\right]  =Z_{t}M_{t+1}$. This variable is unique up to equivalence
$\thicksim_{M_{t+1}}$.
\end{theorem}

Under Assumption \ref{general_assumption}, we can obtain an explicit form of
$Z_{t}$ in the following lemma.

\begin{lemma}
\label{z representation}Under Assumption \ref{general_assumption}, for any
$\mathbb{R}^{K}$ valued and $\mathcal{F}_{t+1}$-measurable random variable $Y$,%

\[
Z_{t}=\sum_{i=1}^{N}\frac{\mathbb{E}\left[  Y1_{\left\{  W_{t+1}%
=e_{i}\right\}  }|\mathcal{F}_{t}\right]  }{P_{t}^{i}}e_{i}^{\ast}.
\]

satisfies $Y-\mathbb{E}\left[  Y|\mathcal{F}_{t}\right]  =Z_{t}M_{t+1}$.
\end{lemma}

\begin{proof}
Without loss of generality, we suppose that $Y$ takes values in $\mathbb{R}$.
For this case, $Z_{t}=\left(  Z_{t}^{1},Z_{t}^{2},...,Z_{t}^{N}\right)
\in\mathbb{R}^{1\times N}$. Note that $\mathbb{E}\left[  Y_{t}|\mathcal{F}%
_{t}\right]  =\sum_{i=1}^{N}\mathbb{E}\left[  Y_{t}1_{\left\{  W_{t+1}%
=e_{i}\right\}  }|\mathcal{F}_{t}\right]  $. By the definition of $M_{t+1}$,
we get%
\[
\sum_{i=1}^{N}Z_{t}^{i}\left(  W_{t+1}^{i}-P_{t}^{i}\right)  =Y_{t}-\sum
_{i=1}^{N}\mathbb{E}\left[  Y_{t}1_{\left\{  W_{t+1}=e_{i}\right\}
}|\mathcal{F}_{t}\right]  .
\]
Multiplying $1_{\left\{  W_{t+1}=e_{j}\right\}  },$ $j\in\left\{
1,2,...,N\right\}  $ on both sides of the above equation and taking
$\mathcal{F}_{t}$-conditional expectation, we obtain the following algebraic
linear equations for $j\in\left\{  1,2,...,N\right\}  $:%

\begin{equation}%
\begin{array}
[c]{cl}
& Z_{t}^{j}\left(  1-P_{t}^{j}\right)  P_{t}^{j}-\sum_{i\neq j,i=1}^{N}%
Z_{t}^{i}P_{t}^{i}P_{t}^{j}\\
= & \mathbb{E}\left[  Y_{t}1_{\left\{  W_{t+1}=e_{j}\right\}  }|\mathcal{F}%
_{t}\right]  -\sum_{i=1}^{N}\mathbb{E}\left[  Y_{t}1_{\left\{  W_{t+1}%
=e_{i}\right\}  }|\mathcal{F}_{t}\right]  P_{t}^{j}.
\end{array}
\label{z_representation_eq}%
\end{equation}
If $\left\{  Z_{t}^{j}\right\}  _{j\in\left\{  1,2,...,N\right\}  }$ is the
(nonunique) solution to (\ref{z_representation_eq}), then $Z_{t}=\left(
Z_{t}^{1},Z_{t}^{2},...,Z_{t}^{N}\right)  $ satisfies $Y-\mathbb{E}\left[
Y|\mathcal{F}_{t}\right]  =Z_{t}M_{t+1}$. Moreover, if $\left\{  Z_{t}%
^{j}\right\}  _{j\in\left\{  1,2,...,N\right\}  }$ and $\left\{
\widetilde{Z}_{t}^{j}\right\}  _{j\in\left\{  1,2,...,N\right\}  }$ are two
solutions to (\ref{z_representation_eq}), then $Z_{t}$ and $\widetilde{Z}_{t}$
are $\thicksim_{M_{t+1}}$equivalence. We rewrite (\ref{z_representation_eq})
as%
\[
\mathbb{E}\left[  Y_{t}1_{\left\{  W_{t+1}=e_{j}\right\}  }|\mathcal{F}%
_{t}\right]  \left(  1-P_{t}^{j}\right)  -\sum_{i\neq j,i=1}^{N}%
\mathbb{E}\left[  Y_{t}1_{\left\{  W_{t+1}=e_{i}\right\}  }|\mathcal{F}%
_{t}\right]  P_{t}^{j}\text{ for }j\in\left\{  1,2,...,N\right\}  \text{.}%
\]
It is easy to check that $Z_{t}^{i}=\frac{\mathbb{E}\left[  Y_{t}1_{\left\{
W_{t+1}=e_{i}\right\}  }|\mathcal{F}_{t}\right]  }{P_{t}^{i}}$ is a solution
to (\ref{z_representation_eq}). Thus, $Z_{t}=\sum_{i=1}^{N}Z_{t}^{i}%
e_{i}^{\ast}$ satisfies $Y_{t}-\mathbb{E}\left[  Y_{t}|\mathcal{F}_{t}\right]
=Z_{t}M_{t+1}$. The proof is completed.
\end{proof}

\begin{remark}
Note that the transition probability $P_{t}^{i}$ appears in the above
expression of $Z_{t}$. In fact, the result is independent of $P_{t}^{i}$ since
they will be canceled by calculation. The value of $Z_{t}$ only depends on the
value of $Y_{t+1}$ in different state.
\end{remark}

For $i\in\left\{  1,2,...,N\right\}  $, define the map $\Lambda_{t}%
^{i}:L\left(  \mathcal{F}_{t+1};\mathbb{R}^{K}\right)  \rightarrow L\left(
\mathcal{F}_{t};\mathbb{R}^{K}\right)  $ by%
\begin{equation}
\Lambda_{t}^{i}\left(  \xi\right)  =\frac{\mathbb{E}\left[  \xi1_{\left\{
W_{t+1}=e_{i}\right\}  }|\mathcal{F}_{t}\right]  }{P_{t}^{i}}
\label{c2_eq_m_lambda}%
\end{equation}
where $\xi\in L\left(  \mathcal{F}_{t+1};\mathbb{R}^{K}\right)  $. Then,
$Z_{t}$ in the above Lemma can be written as $Z_{t}=\sum_{i=1}^{N}\Lambda
_{t}^{i}\left(  Y_{t+1}\right)  e_{i}^{\ast}$.

Now we formulate one kind of BS$\Delta$Es which will be used in the following
sections. For a given sequence $\left\{  U_{t}\right\}  _{t=0}^{T}$, we define
the difference operator $\Delta$ as%
\[
\Delta U_{t}=U_{t+1}-U_{t}.
\]
Consider the following BS$\Delta$E driven by $M$:%

\begin{equation}
\left\{
\begin{array}
[c]{rcl}%
\Delta Y_{t} & = & -f\left(  \omega,t+1,Y_{t+1},Z_{t+1}\right)  +Z_{t}%
M_{t+1},\\
Y_{T} & = & \eta,
\end{array}
\right.  \label{bsde_form1}%
\end{equation}
where $\eta$\ is $\mathcal{F}_{T}$ measurable $\mathbb{R}^{K}$-valued random
variables, $f:\Omega\times\left\{  1,2,...,T\right\}  \times\mathbb{R}%
^{K}\times\mathbb{R}^{K\times N}\longmapsto\mathbb{R}^{K}$ is $\mathcal{F}%
_{t}$-adapted mapping. Additionally we assume that $Z_{T}$ is not included in
function $f$, which means that the generator of BS$\Delta$E (\ref{bsde_form1})
at time $T$ is independent of $Z_{T}$.

\begin{remark}
For discrete time stochastic optimal control problems, the adjoint equations
should just take the form as in (\ref{bsde_form1}).
\end{remark}

\begin{assumption}
\label{BSDE-generator}For any $Z^{1}$ and $Z^{2}\in\mathcal{M}\left(
0,T-1;\mathbb{R}^{K\times N}\right)  $, if $Z^{1}\thicksim_{M}Z^{2}$, then
$f\left(  \omega,t,y,Z_{t}^{1}\right)  =f\left(  \omega,t,y,Z_{t}^{2}\right)
$ for any $\omega\in\Omega$, $t\in\left\{  1,2,...,T-1\right\}  $ and
$y\in\mathbb{R}^{K\times1}$.
\end{assumption}

We have the following theorem.

\begin{theorem}
\label{bsde_result_l}If $f$ satisfies Assumption \ref{BSDE-generator}, then
for any terminal condition $\eta\in L\left(  \mathcal{F}_{T};\mathbb{R}%
^{K}\right)  $,\ BS$\Delta$E (\ref{bsde_form1}) have a unique adapted solution
$\left(  Y,Z\right)  $. Here the uniqueness for $Y$ is in the sense of
indistinguishability and for $Z$ is in the sense of $\thicksim_{M}$ equivalence.
\end{theorem}

\begin{proof}
Taking $\mathcal{F}_{T-1}$-conditional expectation on both sides of
\begin{equation}
\Delta Y_{T-1}=-f\left(  \omega,Y_{T}\right)  +Z_{T-1}M_{T}, \label{BSDE-T-1}%
\end{equation}
we have%
\begin{equation}%
\begin{array}
[c]{rcl}%
Y_{T-1} & = & \mathbb{E}\left[  Y_{T}+f\left(  \omega,Y_{T}\right)
|\mathcal{F}_{t}\right]  .
\end{array}
\label{bsde_m_t+1_proof_0}%
\end{equation}
Adding (\ref{bsde_m_t+1_proof_0}) to (\ref{BSDE-T-1}), we have%
\[
Z_{T-1}M_{T}=Y_{T}+f\left(  \omega,Y_{T}\right)  -\mathbb{E}\left[
Y_{T}+f\left(  \omega,Y_{T}\right)  |\mathcal{F}_{t}\right]  .
\]
By Theorem \ref{mrt_result_cohen_c}, there exists a unique $Z_{T-1}$
satisfying the above equation.

Suppose that we have a unique solution $\left(  Y_{t+1},Z_{t+1}\right)  $ at
time$\ t+1$. At time $t$, by taking $\mathcal{F}_{t}$-conditional expectation,
we deduce%
\begin{equation}
Y_{t}=\mathbb{E}\left[  Y_{t+1}+f\left(  \omega,t+1,Y_{t+1},Z_{t+1}\right)
|\mathcal{F}_{t}\right]  . \label{bsde_proof_1}%
\end{equation}
By Assumption \ref{BSDE-generator}, the solution $Y_{t}$ is unique. Similarly,
we know%
\[%
\begin{array}
[c]{cl}
& Z_{t}M_{t+1}=Y_{t+1}+f\left(  \omega,t+1,Y_{t+1},Z_{t+1}\right)
-\mathbb{E}\left[  Y_{t+1}+f\left(  \omega,t+1,Y_{t+1},Z_{t+1}\right)
|\mathcal{F}_{t}\right]  .
\end{array}
\]
By Theorem \ref{mrt_result_cohen_c}, there still exists a unique $Z_{t}$.

Applying mathematical induction, there exists a unique solution $\left\{
\left(  Y_{t},Z_{t}\right)  \right\}  _{t=0}^{T-1}$.
\end{proof}

\section{Solvability of linear FBS$\Delta$Es}

In this section we study the linear FBS$\Delta$E. For simplicity, we only
consider $1$-dimensional case.

Let

(1) $A_{t},\,B_{t},\,D_{t},\,\widehat{A}_{t},\,\widehat{B}_{t},\,\widehat{D}%
_{t}$ are $\mathcal{F}_{t}$-adapted $1$-dimensional processes;

(2) $G,\,g$ are $\mathcal{F}_{T}$ measurable $1$-dimensional random variables;

(3) $\overline{A}_{t}=\left(  \overline{A}_{t}^{1},\overline{A}_{t}%
^{2},...,\overline{A}_{t}^{N}\right)  ,$\thinspace$\overline{B}_{t}=\left(
\overline{B}_{t}^{1},\overline{B}_{t}^{2},...,\overline{B}_{t}^{N}\right)
,\,\overline{D}_{t}=\left(  \overline{D}_{t}^{1},\overline{D}_{t}%
^{2},...,\overline{D}_{t}^{N}\right)  $ are $\mathcal{F}_{t}$-adapted
processes valued in $\mathbb{R}^{1\times N}$;

(4) $C_{t}=\left(  C_{t}^{1},C_{t}^{2},...,C_{t}^{N}\right)  ^{\ast
},\widehat{C}_{t}=\left(  \widehat{C}_{t}^{1},\widehat{C}_{t}^{2}%
,...,\widehat{C}_{t}^{N}\right)  ^{\ast}$ are $\mathcal{F}_{t}$-adapted
processes valued in $\mathbb{R}^{N\times1}$;

(5) $\overline{C}_{t}=%
\begin{pmatrix}
\overline{C}_{t}^{1} & \overline{C}_{t}^{2} & \cdots & \overline{C}_{t}^{N}%
\end{pmatrix}
$ are $\mathcal{F}_{t}$-adapted processes valued in $\mathbb{R}^{N\times N}$
where $\overline{C}_{t}^{i}=%
\begin{pmatrix}
\overline{C}_{t}^{1i} & \overline{C}_{t}^{2i} & \cdots & \overline{C}_{t}^{Ni}%
\end{pmatrix}
^{\ast}$.

\begin{assumption}
\label{Assu-linear-fbsde} $\mathbf{1}_{N}C_{t}=\mathbf{1}_{N}\widehat{C}%
_{t}=0$, $\widehat{C}_{T}=0$ and $\mathbf{1}_{N}\overline{C}_{t}^{i}=0$ for
$i\in\left\{  1,2,...,N\right\}  $.
\end{assumption}

The above Assumption guarantees that for any $Z_{t},$ $\widetilde{Z}_{t}\in
L\left(  \mathcal{F}_{t};\mathbb{R}^{1\times N}\right)  $ and $t\in\left\{
0,1,2,...,T-1\right\}  $, if $Z_{t}\thicksim_{M_{t+1}}\widetilde{Z}_{t}$, then
$Z_{t}C_{t}=\widetilde{Z}_{t}C_{t},\,Z_{t}\,\widehat{C}_{t}=\widetilde{Z}%
_{t}\,\widehat{C}_{t}$ and$\,Z_{t}\overline{C}_{t}=\widetilde{Z}_{t}%
\overline{C}_{t}.$

Consider the following linear FBS$\Delta$E driven by the martingale difference
process $M$:%
\begin{equation}
\left\{
\begin{array}
[c]{rcl}%
\Delta X_{t} & = & A_{t}X_{t}+B_{t}Y_{t}+Z_{t}C_{t}+D_{t}+\left(
X_{t}\overline{A}_{t}+Y_{t}\overline{B}_{t}+Z_{t}\overline{C}_{t}+\overline
{D}_{t}\right)  M_{t+1},\\
\Delta Y_{t} & = & \widehat{A}_{t+1}X_{t+1}+\widehat{B}_{t+1}Y_{t+1}%
+Z_{t+1}\widehat{C}_{t+1}+\widehat{D}_{t+1}+Z_{t}M_{t+1},\\
X_{0} & = & x_{0},\\
Y_{T} & = & GX_{T}+g.
\end{array}
\right.  \label{fbsde_linear_1d}%
\end{equation}
The solution is a triple $\mathcal{F}_{t}$-adapted processes $\left(
X,Y,Z\right)  $ valued in $\mathbb{R}\times\mathbb{R}\times\mathbb{R}^{1\times
N}$ which satisfies the above equation. For simplicity of notation, we denote
$\mathbf{1}_{N}$ by $\mathbf{1}$ and denote $I_{N\times N}$ by $I$.

Set%
\begin{align*}
\mathcal{A}_{t}  &  =\mathbf{1}\left(  1+A_{t}\right)  +\left(  I-\mathbf{1}%
\mathbb{E}\left[  W_{t+1}^{\ast}|\mathcal{F}_{t}\right]  \right)  \overline
{A}_{t}^{\ast},\\
\mathcal{B}_{t}  &  =\mathbf{1}B_{t}+\left(  I-\mathbf{1}\mathbb{E}\left[
W_{t+1}^{\ast}|\mathcal{F}_{t}\right]  \right)  \overline{B}_{t}^{\ast},\\
\mathcal{C}_{t}  &  =\mathbf{1}C_{t}^{\ast}+\left(  I-\mathbf{1}%
\mathbb{E}\left[  W_{t+1}^{\ast}|\mathcal{F}_{t}\right]  \right)  \overline
{C}_{t}^{\ast},
\end{align*}%
\[
\Lambda_{t}\left(  \phi\right)  =\left(
\begin{array}
[c]{c}%
\Lambda_{t}^{1}\left(  \phi\right) \\
\vdots\\
\Lambda_{t}^{d}\left(  \phi\right)
\end{array}
\right)  ,\text{ }\widetilde{\Lambda}_{t}\left(  \phi\right)  =\left(
\begin{array}
[c]{ccc}%
\Lambda_{t}^{1}\left(  \phi\right)  & \cdots & 0\\
\vdots & \ddots & \vdots\\
0 & \cdots & \Lambda_{t}^{d}\left(  \phi\right)
\end{array}
\right)  ,
\]
where $\phi\in L\left(  \mathcal{F}_{t+1};\mathbb{R}\right)  $, $\Lambda
_{t}^{i}\left(  \cdot\right)  $ is defined in (\ref{c2_eq_m_lambda}). We
define%
\[
\Gamma_{t}\left(  P_{t+1}\right)  =I-\left(  \mathcal{B}_{t}\mathbb{E}\left[
W_{t+1}^{\ast}|\mathcal{F}_{t}\right]  +\mathcal{C}_{t}\right)
\widetilde{\Lambda}_{t}\left(  P_{t+1}\right)  ,
\]
where $P_{t}$ is recursively defined by%
\[
\left\{
\begin{array}
[c]{ccl}%
P_{t} & = & -\widehat{A}_{t}+\theta\left(  \widehat{B}_{t},\widehat{C}%
_{t}\right)  \widetilde{\Lambda}_{t}\left(  P_{t+1}\right)  \left[  \Gamma
_{t}\left(  P_{t+1}\right)  \right]  ^{-1}\mathcal{A}_{t},\\
P_{T} & = & -\widehat{A}_{T}+\left(  1-\widehat{B}_{T}\right)  G
\end{array}
\right.
\]
and $\theta\left(  \widehat{B}_{t},\widehat{C}_{t}\right)  =\left(
1-\widehat{B}_{t}\right)  \mathbb{E}\left[  W_{t+1}^{\ast}|\mathcal{F}%
_{t}\right]  -\widehat{C}_{t}^{\ast}$.

\begin{theorem}
\label{solution theorem for linear fbsde_2}Suppose that Assumption
(\ref{Assu-linear-fbsde}) holds. Then FBS$\Delta$E\thinspace
(\ref{fbsde_linear_1d}) has a unique solution if and only if for any
$t\in\left\{  0,1,2,...,T-1\right\}  $ the $N$-dimensional matrix $\Gamma
_{t}\left(  P_{t+1}\right)  $ is invertible. For this case, the solution to
FBS$\Delta$E\thinspace(\ref{fbsde_linear_1d}) is%
\[
\left\{
\begin{array}
[c]{rcl}%
\Lambda_{t-1}\left(  X_{t}\right)  & = & \left[  \Gamma_{t-1}\left(
P_{t}\right)  \right]  ^{-1}\mathcal{A}_{t-1}X_{t-1}\\
&  & +\left[  \Gamma_{t-1}\left(  P_{t}\right)  \right]  ^{-1}\left[  \left(
\mathcal{B}_{t-1}\mathbb{E}\left[  W_{t}^{\ast}|\mathcal{F}_{t-1}\right]
+\mathcal{C}_{t-1}\right)  \widetilde{\Lambda}_{t-1}\left(  p_{t}\right)
\right]  \mathbf{1},\\
Y_{t} & = & \mathbb{E}\left[  P_{t+1}X_{t+1}+p_{t+1}|\mathcal{F}_{t}\right]
,\\
Z_{t} & = & \Lambda_{t}^{\ast}\left(  P_{t+1}X_{t+1}+p_{t+1}\right)  ,\\
X_{0} & = & x_{0},
\end{array}
\right.
\]
where%
\[
\left\{
\begin{array}
[c]{ccl}%
p_{t} & = & \theta\left(  \widehat{B}_{t},\widehat{C}_{t}\right)
\widetilde{\Lambda}_{t}\left(  P_{t+1}\right)  \left[  \Gamma_{t}\left(
P_{t+1}\right)  \right]  ^{-1}\left[  \left(  \mathcal{B}_{t}\mathbb{E}\left[
W_{t+1}^{\ast}|\mathcal{F}_{t}\right]  +\mathcal{C}_{t}\right)
\widetilde{\Lambda}_{t}\left(  P_{t+1}\right)  \right]  \mathbf{1}\\
&  & +\theta\left(  \widehat{B}_{t},\widehat{C}_{t}\right)  \widetilde{\Lambda
}_{t}\left(  P_{t+1}\right)  \left[  \Gamma_{t}\left(  P_{t+1}\right)
\right]  ^{-1}\mathcal{D}_{t}+\theta\left(  \widehat{B}_{t},\widehat{C}%
_{t}\right)  \Lambda_{t}\left(  p_{t+1}\right)  -\widehat{D}_{t},\\
p_{T} & = & \left(  1-\widehat{B}_{T}\right)  g-\widehat{D}_{T}.
\end{array}
\right.
\]

\end{theorem}

\begin{proof}
Set%
\[
\lambda_{T}=-\widehat{A}_{T}X_{T}+\left(  1-\widehat{B}_{T}\right)
Y_{T}-\widehat{D}_{T}.
\]
Since $Y_{T}=GX_{T}+g$, we have%
\[
\lambda_{T}=\left[  -\widehat{A}_{T}+\left(  1-\widehat{B}_{T}\right)
G\right]  X_{T}+\left(  1-\widehat{B}_{T}\right)  g-\widehat{D}_{T}=P_{T}%
X_{T}+p_{T}.
\]
By the results in section 2, we deduce%
\begin{equation}
\left\{
\begin{array}
[c]{ccl}%
Y_{T-1} & = & \mathbb{E}\left[  \lambda_{T}|\mathcal{F}_{T-1}\right] \\
& = & \mathbb{E}\left[  P_{T}X_{T}|\mathcal{F}_{T-1}\right]  +\mathbb{E}%
\left[  p_{T}|\mathcal{F}_{T-1}\right]  ,\\
Z_{T-1} & = & \Lambda_{T-1}^{\ast}\left(  \lambda_{T}\right) \\
& = & \left(  \Lambda_{T-1}^{1}\left(  P_{T}\right)  \Lambda_{T-1}^{1}\left(
X_{T}\right)  ,...,\Lambda_{T-1}^{N}\left(  P_{T}\right)  \Lambda_{T-1}%
^{N}\left(  X_{T}\right)  \right) \\
&  & +\left(  \Lambda_{T-1}^{1}\left(  p_{T}\right)  ,...,\Lambda_{T-1}%
^{N}\left(  p_{T}\right)  \right)  .
\end{array}
\right.  \label{z_representation_lfbsde2}%
\end{equation}
It is clear that $Y_{T-1}$,\ $Z_{T-1}$ can be represented by $X_{T}$, $P_{T}$
and $p_{T}$. Substituting (\ref{z_representation_lfbsde2}) into the forward
equation, we obtain $N$ equations for $i=1,...,N$ where the $i$-th equation is%
\[%
\begin{array}
[c]{cl}
& \Lambda_{T-1}^{i}\left(  X_{T}\right)  -X_{T-1}\\
= & A_{T-1}X_{T-1}+B_{T-1}\sum\limits_{j=1}^{N}\left[  \Lambda_{T-1}%
^{j}\left(  P_{T}\right)  \Lambda_{T-1}^{j}\left(  X_{T}\right)
+\Lambda_{T-1}^{j}\left(  p_{T}\right)  \right]  e_{j}^{\ast}\mathbb{E}\left[
W_{T}|\mathcal{F}_{T-1}\right] \\
& +\sum\limits_{j=1}^{N}\left[  \Lambda_{T-1}^{j}\left(  P_{T}\right)
\Lambda_{T-1}^{j}\left(  X_{T}\right)  +\Lambda_{T-1}^{j}\left(  p_{T}\right)
\right]  C_{T-1}^{j}+X_{T-1}\overline{A}_{T-1}\left(  e_{i}-\mathbb{E}\left[
W_{T}|\mathcal{F}_{T-1}\right]  \right) \\
& +\sum\limits_{j=1}^{N}\left[  \Lambda_{T-1}^{j}\left(  P_{T}\right)
\Lambda_{T-1}^{j}\left(  X_{T}\right)  +\Lambda_{T-1}^{j}\left(  p_{T}\right)
\right]  e_{j}^{\ast}\mathbb{E}\left[  W_{T}|\mathcal{F}_{T-1}\right]
\overline{B}_{T-1}\left(  e_{i}-\mathbb{E}\left[  W_{T}|\mathcal{F}%
_{T-1}\right]  \right) \\
& +\sum\limits_{j=1}^{N}\left[  \Lambda_{T-1}^{j}\left(  P_{T}\right)
\Lambda_{T-1}^{j}\left(  X_{T}\right)  +\Lambda_{T-1}^{j}\left(  p_{T}\right)
\right]  e_{j}^{\ast}\overline{C}_{T-1}\left(  e_{i}-\mathbb{E}\left[
W_{T}|\mathcal{F}_{T-1}\right]  \right) \\
& +D_{T-1}+\overline{D}_{T-1}\left(  e_{i}-\mathbb{E}\left[  W_{T}%
|\mathcal{F}_{T-1}\right]  \right)
\end{array}
\]
or equivalently, we obtain the following $N$-dimensional linear algebraic
equation of $\Lambda_{T-1}^{1}\left(  X_{T}\right)  ,...,\Lambda_{T-1}%
^{N}\left(  X_{T}\right)  $:%
\begin{equation}%
\begin{array}
[c]{cl}
& \left[  I-\left(  \mathcal{B}_{T-1}\mathbb{E}\left[  W_{T}^{\ast
}|\mathcal{F}_{T-1}\right]  +\mathcal{C}_{T-1}\right)  \widetilde{\Lambda
}_{T-1}\left(  P_{T}\right)  \right]  \Lambda_{T-1}\left(  X_{T}\right) \\
= & \mathcal{A}_{T-1}X_{T-1}+\left[  \left(  \mathcal{B}_{T-1}\mathbb{E}%
\left[  W_{T}^{\ast}|\mathcal{F}_{T-1}\right]  +\mathcal{C}_{T-1}\right)
\widetilde{\Lambda}_{T-1}\left(  P_{T}\right)  \right]  \mathbf{1+}%
\mathcal{D}_{t}.
\end{array}
\label{algebraic_equation_2}%
\end{equation}
It is easy to check that if $\left(  X_{T-1},Y_{T-1},Z_{T-1}\right)  $ is the
solution of FBS$\Delta$E at time $T-1$, then $\Lambda_{T-1}\left(
X_{T}\right)  $ is the solution of linear algebraic equation
(\ref{algebraic_equation_2}). And on the other hand, if $\Lambda_{T-1}\left(
X_{T}\right)  $ is the solution of equation (\ref{algebraic_equation_2}), then
$\left(  X_{T-1},Y_{T-1},Z_{T-1}\right)  $ is the solution of FBS$\Delta$E at
time $T-1$ where $Y_{T}$ and $Z_{T-1}$ are defined by equation
(\ref{z_representation_lfbsde2}). Hence FBS$\Delta$E has a unique solution at
time $T-1$ if and only if the linear algebraic equation
(\ref{algebraic_equation_2}) has a unique solution. The latter statement is
equivalent to say that%
\[
\Gamma_{T-1}\left(  P_{T}\right)  =I-\left(  \mathcal{B}_{T-1}\mathbb{E}%
\left[  W_{T}^{\ast}|\mathcal{F}_{T-1}\right]  +\mathcal{C}_{T-1}\right)
\widetilde{\Lambda}_{T-1}\left(  P_{T}\right)
\]
is a invertible matrix. For this case, $\Lambda_{T-1}\left(  X_{T}\right)  $
can be written as a function of $X_{T-1}$:
\[%
\begin{array}
[c]{cl}%
\Lambda_{T-1}\left(  X_{T}\right)  = & \left[  \Gamma_{T-1}\left(
P_{T}\right)  \right]  ^{-1}\mathcal{A}_{T-1}X_{T-1}\\
& +\left[  \Gamma_{T-1}\left(  P_{T}\right)  \right]  ^{-1}\left[  \left(
\mathcal{B}_{T-1}\mathbb{E}\left[  W_{T}^{\ast}|\mathcal{F}_{T-1}\right]
+\mathcal{C}_{T-1}\right)  \widetilde{\Lambda}_{T-1}\left(  P_{T}\right)
\right]  \mathbf{1}\\
& \mathbf{+}\left[  \Gamma_{T-1}\left(  P_{T}\right)  \right]  ^{-1}%
\mathcal{D}_{t}.
\end{array}
\]
Combining (\ref{z_representation_lfbsde2}), we obtain%
\[
\left\{
\begin{array}
[c]{ccc}%
Y_{T-1} & = & G_{T-1}X_{T-1}+g_{T-1},\\
Z_{T-1} & = & H_{T-1}X_{T-1}+h_{T-1},
\end{array}
\right.
\]
where%
\[
\left\{
\begin{array}
[c]{ccl}%
G_{T-1} & = & \left[  \mathbb{E}\left[  W_{T}^{\ast}|\mathcal{F}_{T-1}\right]
\widetilde{\Lambda}_{T-1}\left(  P_{T}\right)  \right]  \left[  \Gamma
_{T-1}\left(  P_{T}\right)  \right]  ^{-1}\mathcal{A}_{T-1},\\
g_{T-1} & = & \left[  \mathbb{E}\left[  W_{T}^{\ast}|\mathcal{F}_{T-1}\right]
\widetilde{\Lambda}_{T-1}\left(  P_{T}\right)  \right]  \left[  \Gamma
_{T-1}\left(  P_{T}\right)  \right]  ^{-1}\left[  \left(  \mathcal{B}%
_{T-1}\mathbb{E}\left[  W_{T}^{\ast}|\mathcal{F}_{T-1}\right]  +\mathcal{C}%
_{T-1}\right)  \widetilde{\Lambda}_{T-1}\left(  p_{T}\right)  \right]
\mathbf{1}\\
&  & +\left[  \mathbb{E}\left[  W_{T}^{\ast}|\mathcal{F}_{T-1}\right]
\widetilde{\Lambda}_{T-1}\left(  P_{T}\right)  \right]  \left[  \Gamma
_{T-1}\left(  P_{T}\right)  \right]  ^{-1}\mathcal{D}_{t}+\mathbb{E}\left[
W_{T}^{\ast}|\mathcal{F}_{T-1}\right]  \Lambda_{T-1}\left(  p_{T}\right)  ,\\
H_{T-1} & = & \mathcal{A}_{T-1}^{\ast}\left[  \Gamma_{T-1}^{\ast}\left(
P_{T}\right)  \right]  ^{-1}\widetilde{\Lambda}_{T-1}\left(  P_{T}\right)  ,\\
h_{T-1} & = & \left[  \mathbf{1}^{\ast}\left[  \left(  \mathcal{B}%
_{T-1}\mathbb{E}\left[  W_{T}^{\ast}|\mathcal{F}_{T-1}\right]  +\mathcal{C}%
_{T-1}\right)  \widetilde{\Lambda}_{T-1}\left(  p_{T}\right)  \right]  ^{\ast
}\left[  \Gamma_{T-1}^{\ast}\left(  P_{T}\right)  \right]  ^{-1}\right]
\widetilde{\Lambda}_{T-1}\left(  P_{T}\right) \\
&  & \mathbf{+}\mathcal{D}_{t}^{\ast}\left[  \Gamma_{T-1}^{\ast}\left(
P_{T}\right)  \right]  ^{-1}\widetilde{\Lambda}_{T-1}\left(  P_{T}\right)
+\Lambda_{T-1}^{\ast}\left(  p_{T}\right)  .
\end{array}
\right.
\]
Set%
\[%
\begin{array}
[c]{cl}
& \lambda_{T-1}\\
= & -\widehat{A}_{T-1}X_{T-1}+\left(  1-\widehat{B}_{T-1}\right)
Y_{T-1}-Z_{T-1}\widehat{C}_{T-1}-\widehat{D}_{T-1}\\
= & \left[  -\widehat{A}_{T-1}+\left(  1-\widehat{B}_{T-1}\right)
G_{T-1}-H_{T-1}\widehat{C}_{T-1}\right]  X_{T-1}\\
& +\left(  1-\widehat{B}_{T-1}\right)  g_{T-1}-h_{T-1}\widehat{C}%
_{T-1}-\widehat{D}_{T-1}%
\end{array}
\]
and%
\[
\left\{
\begin{array}
[c]{l}%
P_{T-1}=-\widehat{A}_{T-1}+\left(  1-\widehat{B}_{T-1}\right)  G_{T-1}%
-H_{T-1}\widehat{C}_{T-1},\\
p_{T-1}=\left(  1-\widehat{B}_{T-1}\right)  g_{T-1}-h_{T-1}\widehat{C}%
_{T-1}-\widehat{D}_{T-1}.
\end{array}
\right.
\]
Then we have $\lambda_{T-1}=P_{T-1}X_{T-1}+p_{T-1}$. Thus, the result follows
from the backward induction.
\end{proof}

The following corollary can be deduced directly by Theorem
\ref{solution theorem for linear fbsde_2} and will be used in the next section.

\begin{corollary}
\label{solution corollary for linear fbsde_2}For any $D_{t}\in\mathcal{M}%
\left(  0,T-1;\mathbb{R}\right)  $, $\overline{D}_{t}\in\mathcal{M}\left(
0,T-1;\mathbb{R}^{1\times N}\right)  $, $\widehat{D}_{t}\in\mathcal{M}\left(
1,T;\mathbb{R}\right)  $, $g\in L\left(  \mathcal{F}_{T};\mathbb{R}\right)  $,
the linear FBS$\Delta$E%
\begin{equation}
\left\{
\begin{array}
[c]{ccl}%
\Delta X_{t} & = & -Y_{t}+D_{t}+\left(  -Z_{t}\widehat{I}+\overline{D}%
_{t}\right)  M_{t+1},\\
\Delta Y_{t} & = & -X_{t+1}+\widehat{D}_{t+1}+Z_{t}M_{t+1},\\
X_{0} & = & x_{0},\\
Y_{T} & = & X_{T}+g,
\end{array}
\right.  \label{fbsde_linear_1d_special}%
\end{equation}
has a unique solution.
\end{corollary}

\begin{proof}
In this case, the coefficients of the FBS$\Delta$E are%
\begin{align*}
A_{t}  &  =\overline{A}_{t}=\overline{B}_{t}=\widehat{B}_{t}=C_{t}%
=\widehat{C}_{t}=0,\\
\widehat{A}_{t}  &  =B_{t}=-1,\overline{C}_{t}=-\widehat{I},\\
G  &  =1.
\end{align*}
So
\[
\mathcal{A}_{t}=\mathbf{1},\mathcal{B}_{t}=-\mathbf{1},\mathcal{C}%
_{t}=\mathbf{1}\mathbb{E}\left[  W_{t+1}^{\ast}|\mathcal{F}_{t}\right]
-I,\Gamma_{t}\left(  P_{t+1}\right)  =I+\widetilde{\Lambda}_{t}\left(
P_{t+1}\right)
\]
where $P_{t}$ is defined by%
\[%
\begin{array}
[c]{ccl}%
P_{t} & = & 1+\mathbb{E}\left[  P_{t+1}\left(  1+P_{t+1}\right)
^{-1}|\mathcal{F}_{t}\right]  ,\\
P_{T} & = & -\widehat{A}_{T}+\left(  1-\widehat{B}_{T}\right)  G=2.
\end{array}
\]
It is clear that for any $t\in\left\{  1,...,T\right\}  $, $P_{t}$ is
deterministic and $P_{t}>1$. Then $\Gamma_{t-1}\left(  P_{t}\right)  $ is
invertible. By Theorem \ref{solution theorem for linear fbsde_2},
(\ref{fbsde_linear_1d_special}) has a unique solution.
\end{proof}

\begin{remark}
Notice that $Z_{t}\thicksim_{M_{t+1}}Z_{t}\widehat{I}$. So the term
$Z_{t}\widehat{I}$ in FBS$\Delta$E (\ref{fbsde_linear_1d_special}) can be
replaced by $Z_{t}$.
\end{remark}

Now we give two examples to illustrate Theorem
\ref{solution theorem for linear fbsde_2}.

\begin{example}
If FBS$\Delta$E \thinspace(\ref{fbsde_linear_1d}) is only partially coupled,
namely, the forward equation is independent of $Y$ and $Z$, then it is obvious
that there exists a unique solution $X$ to the forward equation and by Theorem
\ref{bsde_result_l} there also exists a unique solution $\left(  Y,Z\right)  $
to the backward equation. We can also obtain the same result by Theorem
\ref{solution theorem for linear fbsde_2} since for this partially coupled
case, we have $B_{t}\equiv0,C_{t}\equiv0,\overline{B}_{t}\equiv0,\overline
{C}_{t}\equiv0$ which leads to $\mathcal{B}_{t}\equiv0$, $\mathcal{C}%
_{t}\equiv0$ and $\Gamma_{t}\left(  P_{t+1}\right)  \equiv I$.
\end{example}

\begin{example}
We suppose that the drift term and diffusion term of the forward equation as
well as the generator of the backward equation are all independent of $Z$.
Then we have $C_{t}\equiv0,\,\overline{C}_{t}\equiv0,\,\widehat{C}_{t}\equiv0$
and%
\[
\Gamma_{t}\left(  P_{t+1}\right)  =I-\left(  \mathcal{B}_{t}\mathbb{E}\left[
W_{t+1}^{\ast}|\mathcal{F}_{t}\right]  \right)  \widetilde{\Lambda}_{t}\left(
P_{t+1}\right)  .
\]
In this case, $P_{t}=-\widehat{A}_{t}+\left(  1-\widehat{B}_{t}\right)  G_{t}$
where%
\begin{align}
G_{t}  &  =\left[  \mathbb{E}\left[  W_{t+1}^{\ast}|\mathcal{F}_{t}\right]
\widetilde{\Lambda}_{t}\left(  P_{t+1}\right)  \right]  \left[  \Gamma
_{t}\left(  P_{t+1}\right)  \right]  ^{-1}\mathcal{A}_{t}\nonumber\\
&  =\left[  \psi_{t}\left(  P_{t+1}\right)  \right]  ^{-1}\left[
\mathbb{E}\left[  W_{t+1}^{\ast}|\mathcal{F}_{t}\right]  \widetilde{\Lambda
}_{t}\left(  P_{t+1}\right)  \right]  \left[  \psi_{t}\left(  P_{t+1}\right)
I+\left(  \mathcal{B}_{t}\mathbb{E}\left[  W_{t+1}^{\ast}|\mathcal{F}%
_{t}\right]  \right)  \widetilde{\Lambda}_{t}\left(  P_{t+1}\right)  \right]
\mathcal{A}_{t}\nonumber\\
&  =\left[  \psi_{t}\left(  P_{t+1}\right)  \right]  ^{-1}\left[  \left(
1+A_{t}\right)  \mathbb{E}\left[  P_{t+1}|\mathcal{F}_{t}\right]
+\overline{A}_{t}\mathbb{E}\left[  P_{t+1}M_{t+1}|\mathcal{F}_{t}\right]
\right]  . \label{fbsde_linear_1d_eq5}%
\end{align}
Let\
\[
\psi_{t}\left(  P_{t+1}\right)  =1-B_{t}\mathbb{E}\left[  P_{t+1}%
|\mathcal{F}_{t}\right]  -\overline{B}_{t}\mathbb{E}\left[  M_{t+1}%
P_{t+1}|\mathcal{F}_{t}\right]  .
\]
We have $\det\left[  \Gamma_{t}\left(  P_{t+1}\right)  \right]  =\psi
_{t}\left(  P_{t+1}\right)  $. By Theorem
\ref{solution theorem for linear fbsde_2}, the solvability condition is
$\psi_{t}\left(  P_{t+1}\right)  \neq0$.

Notice that $Y_{T}=GX_{T}+g$. Then, we suppose that $Y_{t}=G_{t}X_{t}+g_{t}$
holds and try to prove $Y_{t-1}=G_{t-1}X_{t-1}+g_{t-1}$.%
\[%
\begin{array}
[c]{cl}
& Y_{t-1}\\
= & \mathbb{E}\left[  -\widehat{A}_{t}X_{t}+\left(  1-\widehat{B}_{t}\right)
Y_{t}|\mathcal{F}_{t-1}\right] \\
= & \mathbb{E}\left[  \left(  -\widehat{A}_{t}+\left(  1-\widehat{B}%
_{t}\right)  G_{t}\right)  X_{t}+\left(  1-\widehat{B}_{t}\right)
g_{t}|\mathcal{F}_{t-1}\right] \\
= & \mathbb{E}\left[  \left(  -\widehat{A}_{t}+\left(  1-\widehat{B}%
_{t}\right)  G_{t}\right)  \left(  X_{t-1}+A_{t-1}X_{t-1}+B_{t-1}%
Y_{t-1}\right)  |\mathcal{F}_{t-1}\right] \\
& +\mathbb{E}\left[  \left(  -\widehat{A}_{t}+\left(  1-\widehat{B}%
_{t}\right)  G_{t}\right)  \left(  X_{t-1}\overline{A}_{t-1}+Y_{t-1}%
\overline{B}_{t-1}\right)  M_{t}|\mathcal{F}_{t-1}\right] \\
& +\mathbb{E}\left[  \left(  1-\widehat{B}_{t}\right)  g_{t}|\mathcal{F}%
_{t-1}\right] \\
= & \mathbb{E}\left[  -\widehat{A}_{t}+\left(  1-\widehat{B}_{t}\right)
G_{t}|\mathcal{F}_{t-1}\right]  \left(  X_{t-1}+A_{t-1}X_{t-1}+B_{t-1}%
Y_{t-1}\right) \\
& +\left(  X_{t-1}\overline{A}_{t-1}+Y_{t-1}\overline{B}_{t-1}\right)
\mathbb{E}\left[  \left(  -\widehat{A}_{t}+\left(  1-\widehat{B}_{t}\right)
G_{t}\right)  M_{t}|\mathcal{F}_{t-1}\right] \\
& +\mathbb{E}\left[  \left(  1-\widehat{B}_{t}\right)  g_{t}|\mathcal{F}%
_{t-1}\right]  .
\end{array}
\]
We have%
\begin{equation}%
\begin{array}
[c]{cl}
& \left(  1-\mathbb{E}\left[  P_{t}|\mathcal{F}_{t-1}\right]  B_{t-1}%
-\overline{B}_{t-1}\mathbb{E}\left[  P_{t}M_{t}|\mathcal{F}_{t-1}\right]
\right)  Y_{t-1}=\psi_{t-1}\left(  P_{t}\right)  Y_{t-1}\\
= & \left(  1+A_{t-1}\right)  \left[  \mathbb{E}\left[  P_{t}|\mathcal{F}%
_{t-1}\right]  +\overline{A}_{t-1}\mathbb{E}\left[  P_{t}M_{t}|\mathcal{F}%
_{t-1}\right]  \right]  X_{t-1}+\mathbb{E}\left[  \left(  1-\widehat{B}%
_{t}\right)  g_{t}|\mathcal{F}_{t-1}\right]  .
\end{array}
\label{fbsde_linear_1d_eq6}%
\end{equation}
It is obvious that there exists a unique solution $Y_{t-1}$ to equation
(\ref{fbsde_linear_1d_eq6}) if and only if $\psi_{t-1}\left(  P_{t}\right)
\neq0$. Note that%
\[
\left\{
\begin{array}
[c]{l}%
G_{t-1}=\left[  \psi_{t-1}\left(  P_{t}\right)  \right]  ^{-1}\left[  \left(
1+A_{t-1}\right)  \mathbb{E}\left[  P_{t}|\mathcal{F}_{t-1}\right]
+\overline{A}_{t-1}\mathbb{E}\left[  P_{t}M_{t}|\mathcal{F}_{t-1}\right]
\right]  ,\\
g_{t-1}=\left[  \psi_{t-1}\left(  P_{t}\right)  \right]  ^{-1}\mathbb{E}%
\left[  \left(  1-\widehat{B}_{t}\right)  g_{t}|\mathcal{F}_{t-1}\right]  .
\end{array}
\right.
\]
Thus, $Y_{t-1}=G_{t-1}X_{t-1}+g_{t-1}$ and our result coincides with the one
in Theorem \ref{solution theorem for linear fbsde_2}.
\end{example}

\section{Solvability of nonlinear FBS$\Delta$Es}

In this section we consider the following $1$-dimensional nonlinear
FBS$\Delta$E driven by the martingale difference process $M$:%
\begin{equation}
\left\{
\begin{array}
[c]{rcl}%
\Delta X_{t} & = & b\left(  t,X_{t},Y_{t},Z_{t}\widetilde{I}\right)
+\sigma\left(  t,X_{t},Y_{t},Z_{t}\widetilde{I}\right)  M_{t+1},\\
\Delta Y_{t} & = & -f\left(  t+1,X_{t+1},Y_{t+1},Z_{t+1}\widetilde{I}\right)
+Z_{t}M_{t+1},\\
X_{0} & = & x_{0},\\
Y_{T} & = & h\left(  X_{T}\right)  ,
\end{array}
\right.  \label{nonlinear_fbsde_c_method_m0}%
\end{equation}
where%
\begin{align*}
b  &  :\Omega\times\left\{  0,1,2,...,T-1\right\}  \times\mathbb{R}%
\times\mathbb{R}\times\mathbb{R}^{1\times\left(  N-1\right)  }%
\mathbb{\rightarrow R},\\
\sigma &  :\Omega\times\left\{  0,1,2,...,T-1\right\}  \times\mathbb{R}%
\times\mathbb{R}\times\mathbb{R}^{1\times\left(  N-1\right)  }%
\mathbb{\rightarrow R}^{1\times N},\\
f  &  :\Omega\times\left\{  1,2,...,T\right\}  \times\mathbb{R}\times
\mathbb{R}\times\mathbb{R}^{1\times\left(  N-1\right)  }\mathbb{\rightarrow
R},\\
h  &  :\Omega\times\mathbb{R\rightarrow R}%
\end{align*}
are adapted mappings and continuous with respect to $\left(
x,y,z\widetilde{I}\right)  \in\mathbb{R}\times\mathbb{R}\times\mathbb{R}%
^{1\times\left(  N-1\right)  }$. It is worth to pointing out that by Lemma
\ref{m equality}, the value of the coefficients keep unchanged when $Z_{t}$ is
replaced by a $\thicksim_{M_{t+1}}$ equivalent $\tilde{Z}_{t}$.

Define%
\begin{align*}
\lambda &  =\left(  x,y,z\right)  ,\\
\left\vert \lambda\right\vert  &  =\left\vert x\right\vert +\left\vert
y\right\vert +\left\vert z\widetilde{I}\right\vert ,\\
A\left(  t,\lambda\right)   &  =\left(  -f\left(  t,\lambda\right)  ,b\left(
t,\lambda\right)  ,\sigma\left(  t,\lambda\right)  \mathbb{E}\left[
M_{t+1}M_{t+1}^{\ast}|\mathcal{F}_{t}\right]  \right)  ,\\
\left\vert A\left(  t,\lambda\right)  \right\vert  &  =\left\vert f\left(
t,\lambda\right)  \right\vert +\left\vert b\left(  t,\lambda\right)
\right\vert +\left\vert \sigma\left(  t,\lambda\right)  \mathbb{E}\left[
M_{t+1}M_{t+1}^{\ast}|\mathcal{F}_{t}\right]  \right\vert .
\end{align*}

\begin{assumption}
\label{assumption for nonlinear fbsde}(i) The coefficients are
uniform\ Lipschitz continuous with respect to $\lambda$, i.e. there exists a
constant $c_{1}>0$, such that for any $t\in\left\{  0,1,...,T-1\right\}  $,%
\begin{align*}
\left\vert A\left(  t,\lambda\right)  -A\left(  t,\lambda^{^{\prime}}\right)
\right\vert  &  \leq c_{1}\left\vert \lambda-\lambda^{^{\prime}}\right\vert
,\\
\left\vert h\left(  x\right)  -h\left(  x^{^{\prime}}\right)  \right\vert  &
\leq c_{1}\left\vert x-x^{^{\prime}}\right\vert .
\end{align*}

Moreover, $\left\vert f\left(  T,x,y,z\widetilde{I}\right)  -f\left(
T,x^{^{\prime}},y^{^{\prime}},z^{^{\prime}}\widetilde{I}\right)  \right\vert
\leq c_{1}\left(  \left\vert x-x^{^{\prime}}\right\vert +\left\vert
y-y^{^{\prime}}\right\vert \right)  ,$

(ii) The coefficients satisfy monotone condition, i.e. there exists a constant
$c_{2}>0$, such that

when $t\in\left\{  1,...,T-1\right\}  $,%
\[
\left\langle A\left(  t,\lambda\right)  -A\left(  t,\lambda^{^{\prime}%
}\right)  ,\lambda-\lambda^{^{\prime}}\right\rangle \leq-c_{2}\left\vert
\lambda-\lambda^{^{\prime}}\right\vert ^{2},
\]

when $t=T$,%
\[
\left\langle -f\left(  T,x,y,z\widetilde{I}\right)  +f\left(  T,x^{^{\prime}%
},y^{^{\prime}},z^{^{\prime}}\widetilde{I}\right)  ,x-x^{^{\prime}%
}\right\rangle \leq-c_{2}\left\vert x-x^{^{\prime}}\right\vert ^{2},
\]

when $t=0$,%
\[%
\begin{array}
[c]{cl}
& \left\langle b\left(  0,\lambda\right)  -b\left(  0,\lambda^{^{\prime}%
}\right)  ,y-y^{^{\prime}}\right\rangle \\
& +\left\langle \left(  \sigma\left(  0,\lambda\right)  -\sigma\left(
0,\lambda^{^{\prime}}\right)  \right)  \mathbb{E}\left[  M_{1}M_{1}^{\ast
}|\mathcal{F}_{0}\right]  ,z-z^{^{\prime}}\right\rangle \\
\leq & -c_{2}\left[  \left\vert y-y^{^{\prime}}\right\vert ^{2}+\left\vert
\left(  z-z^{^{\prime}}\right)  \widetilde{I}\right\vert ^{2}\right]
\end{array}
\]

and%
\[
\left\langle h\left(  x\right)  -h\left(  x^{^{\prime}}\right)  ,x-x^{^{\prime
}}\right\rangle \geq c_{2}\left\vert x-x^{^{\prime}}\right\vert ^{2}.
\]

\end{assumption}

Then we have the following existence and uniqueness theorem for FBS$\Delta
$E\thinspace(\ref{nonlinear_fbsde_c_method_m0}). Note that the uniqueness for
$X$ and $Y$ is in the sense of indistinguishability and for $Z$ is in the
sense of $\thicksim_{M}$ equivalence.

\begin{theorem}
\label{c4_th_m_fbsde_ex}Under Assumption \ref{assumption for nonlinear fbsde},
FBS$\Delta$E\thinspace(\ref{nonlinear_fbsde_c_method_m0}) has a unique adapted
solution $\left(  X,Y,Z\right)  \in\mathcal{M}\left(  0,T;\mathbb{R}\right)
\times\mathcal{M}\left(  0,T;\mathbb{R}\right)  \times\mathcal{M}\left(
0,T-1;\mathbb{R}^{1\times N}\right)  $.
\end{theorem}

\begin{proof}
We first prove the uniqueness.

Suppose that $\Lambda=\left(  X,Y,Z\right)  $ and $\Lambda^{^{\prime}}=\left(
X^{^{\prime}},Y^{^{\prime}},Z^{^{\prime}}\right)  $ are two solutions for
FBS$\Delta$E\thinspace(\ref{nonlinear_fbsde_c_method_m0}). Define%
\[
\left(  \widehat{X},\widehat{Y},\widehat{Z}\right)  =\left(  X-X^{^{\prime}%
},Y-Y^{^{\prime}},Z-Z^{^{\prime}}\right)
\]
and%
\begin{align*}
\widehat{b}\left(  t\right)   &  =b\left(  t,X_{t},Y_{t},Z_{t}\widetilde{I}%
\right)  -b\left(  t,X_{t}^{\prime},Y_{t}^{\prime},Z_{t}^{^{\prime}%
}\widetilde{I}\right)  ,\\
\widehat{\sigma}\left(  t\right)   &  =\sigma\left(  t,X_{t},Y_{t}%
,Z_{t}\widetilde{I}\right)  -\sigma\left(  t,X_{t}^{\prime},Y_{t}^{\prime
},Z_{t}^{\prime}\widetilde{I}\right)  ,\\
\widehat{f}\left(  t\right)   &  =f\left(  t,X_{t},Y_{t},Z_{t}\widetilde{I}%
\right)  -f\left(  t,X_{t}^{\prime},Y_{t}^{^{\prime}},Z_{t}^{^{\prime}%
}\widetilde{I}\right)  .
\end{align*}
For $t\in\left\{  0,1,...,T-1\right\}  $, we have%
\[%
\begin{array}
[c]{cl}
& \Delta\left(  \widehat{X}_{t}\widehat{Y}_{t}\right) \\
= & \widehat{X}_{t+1}\left(  \Delta\widehat{Y}_{t}\right)  +\left(
\Delta\widehat{X}_{t}\right)  \widehat{Y}_{t}\\
= & -\widehat{X}_{t+1}\widehat{f}\left(  t+1\right)  +\left(  \widehat{X}%
_{t}+\Delta\widehat{X}_{t}\right)  \widehat{Z}_{t}M_{t+1}\\
& +\widehat{b}\left(  t\right)  \widehat{Y}_{t}+\widehat{\sigma}\left(
t\right)  M_{t+1}\widehat{Y}_{t}\\
= & -\widehat{X}_{t+1}\widehat{f}\left(  t+1\right)  +\left(  \widehat{\sigma
}\left(  t\right)  M_{t+1}\right)  \widehat{Z}_{t}M_{t+1}+\widehat{b}\left(
t\right)  \widehat{Y}_{t}+\Phi_{t}\\
= & -\widehat{X}_{t+1}\widehat{f}\left(  t+1\right)  +\left\langle
\widehat{\sigma}\left(  t\right)  M_{t+1}M_{t+1}^{\ast},\widehat{Z}%
_{t}\right\rangle +\widehat{b}\left(  t\right)  \widehat{Y}_{t}+\Phi_{t}%
\end{array}
\]
where%
\[
\Phi_{t}=\left(  X_{t}+\widehat{b}\left(  t\right)  \right)  \widehat{Z}%
_{t}M_{t+1}+\widehat{\sigma}\left(  t\right)  M_{t+1}\widehat{Y}_{t}.
\]
Then,%
\[%
\begin{array}
[c]{cl}
& \mathbb{E}\left[  \left(  X_{T}-X_{T}^{\prime}\right)  \left[  h\left(
X_{T}\right)  -h\left(  X_{T}^{\prime}\right)  \right]  \right] \\
= & \mathbb{E}\left[  \widehat{X}_{T}\widehat{Y}_{T}\right] \\
= & \mathbb{E}\sum\limits_{t=0}^{T-1}\Delta\left(  \widehat{X}_{t}%
\widehat{Y}_{t}\right) \\
= & \mathbb{E}\left[  \sum\limits_{t=0}^{T-2}\widehat{X}_{t+1}\widehat{f}%
\left(  t+1\right)  +\sum\limits_{t=0}^{T-1}\left\langle \widehat{\sigma
}\left(  t\right)  M_{t+1}M_{t+1}^{\ast},\widehat{Z}_{t}\right\rangle \right]
\\
& +\mathbb{E}\left[  \sum\limits_{t=0}^{T-1}\widehat{b}\left(  t\right)
\widehat{Y}_{t}-\widehat{X}_{T}\widehat{f}\left(  T\right)  \right] \\
= & \mathbb{E}\left[  \sum\limits_{t=1}^{T-1}\left\langle A\left(
t,\Lambda\right)  -A\left(  t,\Lambda^{^{\prime}}\right)  ,\Lambda
-\Lambda^{^{\prime}}\right\rangle \right] \\
& +\mathbb{E}\left[  \widehat{b}\left(  0\right)  \widehat{Y}_{0}+\left\langle
\widehat{\sigma}\left(  0\right)  M_{1}M_{1}^{\ast},\widehat{Z}_{0}%
\right\rangle -\widehat{X}_{T}\widehat{f}\left(  T\right)  \right]  .
\end{array}
\]
By the monotone condition,%
\[%
\begin{array}
[c]{cl}
& c_{2}\mathbb{E}\left\vert X_{T}-X_{T}^{\prime}\right\vert ^{2}\\
\leq & \mathbb{E}\left[  \left(  X_{T}-X_{T}^{\prime}\right)  \left[  h\left(
X_{T}\right)  -h\left(  X_{T}^{\prime}\right)  \right]  \right] \\
= & \mathbb{E}\left[  \sum\limits_{t=1}^{T-1}\left\langle A\left(
t,\Lambda\right)  -A\left(  t,\Lambda^{^{\prime}}\right)  ,\Lambda
-\Lambda^{^{\prime}}\right\rangle \right] \\
& +\mathbb{E}\left[  \widehat{b}\left(  0\right)  \widehat{Y}_{0}+\left\langle
\widehat{\sigma}\left(  0\right)  \mathbb{E}\left[  M_{1}M_{1}^{\ast
}|\mathcal{F}_{0}\right]  ,\widehat{Z}_{0}\right\rangle -\widehat{X}%
_{T}\widehat{f}\left(  T\right)  \right] \\
\leq & -c_{2}\mathbb{E}\left[  \sum\limits_{t=0}^{T}\left\vert X_{t}%
-X_{t}^{^{\prime}}\right\vert ^{2}+\sum\limits_{t=0}^{T-1}\left\vert
Y_{t}-Y_{t}^{^{\prime}}\right\vert ^{2}+\sum\limits_{t=0}^{T-1}\left\vert
\left(  Z_{t}-Z_{t}^{^{\prime}}\right)  \widetilde{I}\right\vert ^{2}\right]
\end{array}
\]
which yields that%
\[
\mathbb{E}\left[  \sum_{t=0}^{T}\left\vert X_{t}-X_{t}^{^{\prime}}\right\vert
^{2}+\sum_{t=0}^{T-1}\left\vert Y_{t}-Y_{t}^{^{\prime}}\right\vert ^{2}%
+\sum_{t=0}^{T-1}\left\vert \left(  Z_{t}-Z_{t}^{^{\prime}}\right)
\widetilde{I}\right\vert ^{2}\right]  =0.
\]
Thus, we deduce that $X=X^{^{\prime}}$, $Y=Y^{^{\prime}}$, $Z\thicksim
_{M}Z^{^{\prime}}$.

Now we prove the existence.

Introduce the following family of FBS$\Delta$Es parameterized by $\alpha
\in\left[  0,1\right]  $:%
\begin{equation}
\left\{
\begin{array}
[c]{rcl}%
\Delta X_{t} & = & b^{\alpha}\left(  t,X_{t},Y_{t},Z_{t}\widetilde{I}\right)
+b_{0}\left(  t\right)  +\left[  \sigma^{\alpha}\left(  t,X_{t},Y_{t}%
,Z_{t}\widetilde{I}\right)  +\sigma_{0}\left(  t\right)  \right]  M_{t+1},\\
\Delta Y_{t} & = & -f^{\alpha}\left(  t+1,X_{t+1},Y_{t+1},Z_{t+1}%
\widetilde{I}\right)  -f_{0}\left(  t+1\right)  +Z_{t}M_{t+1},\\
X_{0} & = & x_{0},\\
Y_{T} & = & h^{\alpha}\left(  X_{T}\right)  +h_{0}%
\end{array}
\right.  \label{nonlinear_fbsde_c_method_m1}%
\end{equation}
where%
\begin{align*}
b^{\alpha}\left(  t,x,y,z\right)   &  =\alpha b\left(  t,x,y,z\widetilde{I}%
\right)  +\left(  1-\alpha\right)  \left(  -y\right)  ,\\
\sigma^{\alpha}\left(  t,x,y,z\right)   &  =\alpha\sigma\left(
t,x,y,z\widetilde{I}\right)  +\left(  1-\alpha\right)  \left(  -z\right)  ,\\
f^{\alpha}\left(  t,x,y,z\right)   &  =\alpha f\left(  t,x,y,z\widetilde{I}%
\right)  +\left(  1-\alpha\right)  x,\\
h^{a}\left(  x\right)   &  =\alpha h\left(  x\right)  +\left(  1-\alpha
\right)  x.
\end{align*}
When $\alpha=1$, FBS$\Delta$E (\ref{nonlinear_fbsde_c_method_m1}) becomes
FBS$\Delta$E (\ref{nonlinear_fbsde_c_method_m0}). On the other hand, when
$\alpha=0$, FBS$\Delta$E (\ref{nonlinear_fbsde_c_method_m1}) becomes a linear
equation:%
\begin{equation}
\left\{
\begin{array}
[c]{rcl}%
\Delta X_{t} & = & -Y_{t}+b_{0}\left(  t\right)  +\left[  -Z_{t}+\sigma
_{0}\left(  t\right)  \right]  M_{t+1},\\
\Delta Y_{t} & = & -X_{t+1}-f_{0}\left(  t+1\right)  +Z_{t}M_{t+1},\\
X_{0} & = & x_{0},\\
Y_{T} & = & X_{T}+h_{0}.
\end{array}
\right.  \label{nonlinear_fbsde_c_method_m2}%
\end{equation}
By Corollary \ref{solution corollary for linear fbsde_2}, we know
that\ (\ref{nonlinear_fbsde_c_method_m2}) has a unique solution.

To complete the proof, we need the following lemma.

\begin{lemma}
\label{nonlinear_fbsde_c_method_mlemma}Suppose that there exists an
$\alpha_{0}\in\left[  0,1\right)  $ such that for any \thinspace$\left(
b_{0}\left(  \cdot\right)  ,\sigma_{0}\left(  \cdot\right)  ,f_{0}\left(
\cdot\right)  \right)  \in\mathcal{M}\left(  0,T;\mathbb{R}\times
\mathbb{R}^{1\times N}\times\mathbb{R}\right)  $, $h_{0}\in L\left(
\mathcal{F}_{T};\mathbb{R}\right)  $, (\ref{nonlinear_fbsde_c_method_m1}) has
a unique solution. Then there exists a $\delta_{0}\in\left(  0,1\right)  $,
which only depends on $c_{1},$ $c_{2}$ and $T$, such that for any $\alpha
\in\left[  \alpha_{0},\alpha_{0}+\delta_{0}\right]  $, $\left(  b_{0}\left(
\cdot\right)  ,\sigma_{0}\left(  \cdot\right)  ,f_{0}\left(  \cdot\right)
\right)  \in\mathcal{M}\left(  0,T;\mathbb{R}\times\mathbb{R}^{1\times
N}\times\mathbb{R}\right)  $ and $h_{0}\in L\left(  \mathcal{F}_{T}%
;\mathbb{R}\right)  $, (\ref{nonlinear_fbsde_c_method_m1}) has a unique solution.
\end{lemma}

\begin{proof}
Note that%
\begin{align*}
b^{\alpha_{0}+\delta}\left(  t,x,y,z\widetilde{I}\right)   &  =b^{\alpha_{0}%
}\left(  t,x,y,z\widetilde{I}\right)  +\delta\left(  y+b\left(
t,x,y,z\widetilde{I}\right)  \right)  ,\\
\sigma^{\alpha_{0}+\delta}\left(  t,x,y,z\widetilde{I}\right)   &
=\sigma^{\alpha_{0}}\left(  t,x,y,z\widetilde{I}\right)  +\delta\left(
z+\sigma\left(  t,x,y,z\widetilde{I}\right)  \right)  ,\\
f^{\alpha_{0}+\delta}\left(  t,x,y,z\widetilde{I}\right)   &  =f^{\alpha_{0}%
}\left(  t,x,y,z\widetilde{I}\right)  +\delta\left(  -x+f\left(
t,x,y,z\widetilde{I}\right)  \right)  ,\\
h^{\alpha_{0}+\delta}\left(  x\right)   &  =h^{\alpha_{0}}\left(  x\right)
+\delta\left(  -x+h\left(  x\right)  \right)  .
\end{align*}
Let $\Lambda^{i}=\left(  X^{i},Y^{i},Z^{i}\right)  $ and $\Lambda^{0}=0$. Then
we solve iteratively the following equations:%
\[
\left\{
\begin{array}
[c]{rcl}%
\Delta X_{t}^{i+1} & = & b^{\alpha_{0}}\left(  t,\Lambda_{t}^{i+1}\right)
+\delta\left(  Y_{t}^{i}+b\left(  t,\Lambda_{t}^{i}\right)  \right)
+b_{0}\left(  t\right) \\
&  & +\left[  \sigma^{\alpha_{0}}\left(  t,\Lambda_{t}^{i+1}\right)
+\delta\left(  Z_{t}^{i}+\sigma\left(  t,\Lambda_{t}^{i}\right)  \right)
+\sigma_{0}\left(  t\right)  \right]  M_{t+1},\\
\Delta Y_{t}^{i+1} & = & -f^{\alpha_{0}}\left(  t+1,\Lambda_{t+1}%
^{i+1}\right)  -\delta\left(  -X_{t+1}^{i}+f\left(  t+1,\Lambda_{t+1}%
^{i}\right)  \right) \\
&  & -f_{0}\left(  t+1\right)  +Z_{t}^{i+1}M_{t+1},\\
X_{0}^{i+1} & = & x_{0},\\
Y_{T}^{i+1} & = & h^{\alpha_{0}}\left(  X_{T}^{i+1}\right)  +\delta\left(
-X_{T}^{i}+h\left(  X_{T}^{i}\right)  \right)  +h_{0}.
\end{array}
\right.
\]
By calculating $\mathbb{E}\sum_{t=0}^{T-1}\Delta\left(  \widehat{X}_{t}%
^{i+1}\widehat{Y}_{t}^{i+1}\right)  $ as in the proof of the uniqueness part,
we have%
\[%
\begin{array}
[c]{cl}
& \mathbb{E}\left[  \widehat{X}_{T}^{i+1}\left(  h^{\alpha_{0}}\left(
X_{T}^{i+1}\right)  -h^{\alpha_{0}}\left(  X_{T}^{i}\right)  \right)
+\widehat{X}_{T}^{i+1}\delta\left(  -\left(  X_{T}^{i}-X_{T}^{i-1}\right)
+h\left(  X_{T}^{i}\right)  -h\left(  X_{T}^{i-1}\right)  \right)  \right] \\
= & \mathbb{E}\left[  \widehat{X}_{T}^{i+1}\widehat{Y}_{T}^{i+1}\right] \\
= & \mathbb{E}\left[  \varphi_{0}+\sum_{t=1}^{T-1}\varphi_{t}+\varphi
_{T}+\delta\left(  \psi_{0}+\sum_{t=1}^{T-1}\psi_{t}+\psi_{T}\right)  \right]
,
\end{array}
\]
where%
\[
\left\{
\begin{array}
[c]{ccl}%
\varphi_{0} & = & \alpha_{0}\left[  b\left(  0,\Lambda_{0}^{i+1}\right)
-b\left(  0,\Lambda_{0}^{i}\right)  \right]  \widehat{Y}_{0}^{i+1}-\left(
1-\alpha_{0}\right)  \left(  \widehat{Y}_{0}^{i+1}\right)  ^{2}\\
&  & +\alpha_{0}\left(  \sigma\left(  0,\Lambda_{0}^{i+1}\right)
-\sigma\left(  0,\Lambda_{0}^{i}\right)  \right)  \mathbb{E}\left[  M_{1}%
M_{1}^{\ast}|\mathcal{F}_{0}\right]  \widehat{Z}_{0}^{i+1}-\left(
1-\alpha_{0}\right)  \left(  \widehat{Z}_{0}^{i+1}M_{1}\right)  ^{2},\\
\varphi_{t} & = & -\left(  1-\alpha_{0}\right)  \left(  \widehat{X}_{t}%
^{i+1}\right)  ^{2}-\left(  1-\alpha_{0}\right)  \left(  \widehat{Y}_{t}%
^{i+1}\right)  ^{2}-\left(  1-\alpha_{0}\right)  \left(  \widehat{Z}_{t}%
^{i+1}M_{t+1}\right)  ^{2}\\
&  & +\alpha_{0}\left\langle A\left(  t,\Lambda_{t}^{i+1}\right)  -A\left(
t,\Lambda_{t}^{i}\right)  ,\widehat{\Lambda}_{t}^{i+1}\right\rangle ,\\
\varphi_{T} & = & \alpha_{0}\left[  -f\left(  X_{T}^{i+1},Y_{T}^{i+1}\right)
+f\left(  X_{T}^{i},Y_{T}^{i}\right)  \right]  \widehat{X}_{T}^{i+1}-\left(
1-\alpha_{0}\right)  \left(  \widehat{X}_{T}^{i+1}\right)  ^{2},\\
\psi_{0} & = & \left[  b\left(  0,\Lambda_{0}^{i}\right)  -b\left(
0,\Lambda_{0}^{i-1}\right)  \right]  \widehat{Y}_{0}^{i+1}+\widehat{Y}_{0}%
^{i}\widehat{Y}_{0}^{i+1}\\
&  & +\left\langle \left(  \sigma\left(  0,\Lambda_{0}^{i}\right)
-\sigma\left(  0,\Lambda_{0}^{i-1}\right)  \right)  M_{1},\widehat{Z}%
_{0}^{i+1}M_{1}\right\rangle +\left\langle \widehat{Z}_{0}^{i}M_{1}%
,\widehat{Z}_{0}^{i+1}M_{1}\right\rangle ,\\
\psi_{t} & = & \widehat{X}_{t}^{i}\widehat{X}_{t}^{i+1}+\widehat{Y}_{t}%
^{i}\widehat{Y}_{t}^{i+1}+\left(  \widehat{Z}_{t}^{i}M_{t+1}\right)  \left(
\widehat{Z}_{t}^{i+1}M_{t+1}\right) \\
&  & +\left\langle A\left(  t,\Lambda_{t}^{i}\right)  -A\left(  t,\Lambda
_{t}^{i-1}\right)  ,\widehat{\Lambda}_{t}^{i+1}\right\rangle ,\\
\psi_{T} & = & \left[  -f\left(  X_{T}^{i},Y_{T}^{i}\right)  +f\left(
X_{T}^{i\_1},Y_{T}^{i-1}\right)  \right]  \widehat{X}_{T}^{i+1}+\widehat{X}%
_{T}^{i}\widehat{X}_{T}^{i+1}%
\end{array}
\right.
\]
and $\widehat{\Lambda}_{t}^{i}=\Lambda_{t}^{i}-\Lambda_{t}^{i-1}$. By
Proposition \ref{norm equality}, we have $\underline{L}\mathbb{E}\left\vert
\widehat{Z}_{t}\widetilde{I}\right\vert ^{2}\leq\mathbb{E}\left(
\widehat{Z}_{t}M_{t+1}\right)  ^{2}\leq\overline{L}\mathbb{E}\left\vert
\widehat{Z}_{t}\widetilde{I}\right\vert ^{2}$. Set\ $\underline{c}%
=\min\left\{  \underline{L},c_{2}\right\}  $, $\overline{c}=\max\left\{
\overline{L},1\right\}  $. Then we obtain%
\[%
\begin{array}
[c]{cl}
& \mathbb{E}\left[  2\left\vert \widehat{X}_{T}^{i+1}\right\vert ^{2}%
+\sum\limits_{t=0}^{T-1}\left\vert \widehat{\Lambda}_{t}^{i+1}\right\vert
^{2}\right] \\
\leq & \frac{\delta\left(  \overline{c}+c_{1}\right)  }{\underline{c}}\left(
2\mathbb{E}\left\vert \widehat{X}_{T}^{i}\right\vert \left\vert \widehat{X}%
_{T}^{i+1}\right\vert +\mathbb{E}\sum\limits_{t=0}^{T-1}\left\vert
\widehat{\Lambda}_{t}^{i}\right\vert \left\vert \widehat{\Lambda}_{t}%
^{i+1}\right\vert \right) \\
\leq & \frac{1}{2}\left(  \frac{\delta\left(  \overline{c}+c_{1}\right)
}{\underline{c}}\right)  ^{2}\left(  2\mathbb{E}\left\vert \widehat{X}_{T}%
^{i}\right\vert ^{2}+\mathbb{E}\sum\limits_{t=0}^{T-1}\left\vert
\widehat{\Lambda}_{t}^{i}\right\vert ^{2}\right)  +\frac{1}{2}\left(
2\mathbb{E}\left\vert \widehat{X}_{T}^{i+1}\right\vert ^{2}+\mathbb{E}%
\sum\limits_{t=0}^{T-1}\left\vert \widehat{\Lambda}_{t}^{i+1}\right\vert
^{2}\right)
\end{array}
\]
and%
\begin{equation}
2\mathbb{E}\left\vert \widehat{X}_{T}^{i+1}\right\vert ^{2}+\mathbb{E}%
\sum_{t=0}^{T-1}\left\vert \widehat{\Lambda}_{t}^{i+1}\right\vert ^{2}%
\leq\left(  \frac{\delta\left(  \overline{c}+c_{1}\right)  }{\underline{c}%
}\right)  ^{2}\left(  2\mathbb{E}\left\vert \widehat{X}_{T}^{i}\right\vert
^{2}+\mathbb{E}\sum_{t=0}^{T-1}\left\vert \widehat{\Lambda}_{t}^{i}\right\vert
^{2}\right)  . \label{nonlinear_fbsde_c_method_meq3}%
\end{equation}
We also have the following estimation%
\begin{equation}
\mathbb{E}\left\vert \widehat{X}_{T}^{i}\right\vert ^{2}=\mathbb{E}\left\vert
\sum_{t=0}^{T-1}\Delta\widehat{X}_{t}^{i}\right\vert ^{2}\leq T\left\vert
\Delta\widehat{X}_{t}^{i}\right\vert ^{2}\leq c_{3}\mathbb{E}\sum_{t=0}%
^{T-1}\left(  \left\vert \widehat{\Lambda}_{t}^{i}\right\vert ^{2}+\left\vert
\widehat{\Lambda}_{t}^{i-1}\right\vert ^{2}\right)
\label{nonlinear_fbsde_c_method_meq4}%
\end{equation}
where $c_{3}>0$ only depends on $c_{1}$ and $T$.

Combining with (\ref{nonlinear_fbsde_c_method_meq3}) and
(\ref{nonlinear_fbsde_c_method_meq4}), there exists a constant\ $c_{4}>0$
which only depends on $c_{1},$ $c_{2}$ and $T$, such that%
\[
\mathbb{E}\sum_{t=0}^{T-1}\left\vert \widehat{\Lambda}_{t}^{i+1}\right\vert
^{2}\leq c_{4}\delta^{2}\left(  \mathbb{E}\sum_{t=0}^{T-1}\left\vert
\widehat{\Lambda}_{t}^{i}\right\vert ^{2}+\mathbb{E}\sum_{t=0}^{T-1}\left\vert
\widehat{\Lambda}_{t}^{i-1}\right\vert ^{2}\right)  .
\]
So there exists $\delta_{0}\in\left(  0,1\right)  $ which only depends on
$c_{1},$ $c_{2}$ and $T$, such that for $0<\delta\leq\delta_{0}$,%
\[
\mathbb{E}\sum_{t=0}^{T-1}\left\vert \widehat{\Lambda}_{t}^{i+1}\right\vert
^{2}\leq\frac{1}{4}\mathbb{E}\sum_{t=0}^{T-1}\left\vert \widehat{\Lambda}%
_{t}^{i}\right\vert ^{2}+\frac{1}{8}\mathbb{E}\sum_{t=0}^{T-1}\left\vert
\widehat{\Lambda}_{t}^{i-1}\right\vert ^{2},\forall i\geq1.
\]
By Lemma 4.1 in \cite{hp95}, $\left\{  \Lambda_{t}^{i}\right\}  _{t=0}^{T-1}$
is a Cauchy sequence in $\mathcal{M}\left(  0,T-1;\mathbb{R}\times
\mathbb{R}\times\mathbb{R}^{1\times N}\right)  $. Taking $\left\{  \Lambda
_{t}\right\}  _{t=0}^{T-1}=\lim_{i\rightarrow\infty}\left\{  \Lambda_{t}%
^{i}\right\}  _{t=0}^{T-1}$, it is easy to check that for $0<\delta\leq
\delta_{0}$, $\Lambda=\left(  X,Y,Z\right)  $ is the solution to FBS$\Delta$E
(\ref{nonlinear_fbsde_c_method_m1}) with $\alpha=\alpha_{0}+\delta$. \ This
completes the proof of this lemma.
\end{proof}

Now we can finish the proof of the existence. When $\alpha=0$, for any
$\left(  b_{0}\left(  \cdot\right)  ,\sigma_{0}\left(  \cdot\right)
,f_{0}\left(  \cdot\right)  \right)  \in\mathcal{M}\left(  0,T;\mathbb{R}%
\times\mathbb{R}^{1\times N}\times\mathbb{R}\right)  $ and $h_{0}\in L\left(
\mathcal{F}_{T};\mathbb{\mathbb{R}}\right)  $, there exists a unique solution
to FBS$\Delta$E\ (\ref{nonlinear_fbsde_c_method_m1}). By Lemma
\ref{nonlinear_fbsde_c_method_mlemma}, there exists a constant $\delta_{0}$
which only depends on $c_{1},$ $c_{2}$ and $T$, such that for any $\left(
b_{0}\left(  \cdot\right)  ,\sigma_{0}\left(  \cdot\right)  ,f_{0}\left(
\cdot\right)  \right)  \in\mathcal{M}\left(  0,T;\mathbb{R}\times
\mathbb{R}^{1\times N}\times\mathbb{R}\right)  $ and $h_{0}\in L\left(
\mathcal{F}_{T};\mathbb{\mathbb{R}}\right)  $, FBS$\Delta$E
(\ref{nonlinear_fbsde_c_method_m1}) has a unique solution for\ $\alpha
\in\left[  0,\delta_{0}\right]  ,$ $\left[  \delta_{0},2\delta_{0}\right]
,...$ It turns out that FBS$\Delta$E (\ref{nonlinear_fbsde_c_method_m1}) has a
unique solution when $\alpha=1$. Taking $b_{0}\left(  \cdot\right)
=\sigma_{0}\left(  \cdot\right)  =f_{0}\left(  \cdot\right)  =0$, $h_{0}=0$,
we deduce that the solution of FBS$\Delta$E (\ref{nonlinear_fbsde_c_method_m0}%
) exists. This completes the proof.
\end{proof}

\end{document}